\documentclass[11pt]{article}
\usepackage{amsfonts}
\usepackage{color}
\usepackage{amsthm}
\usepackage{graphicx}
\usepackage{hyperref}
\usepackage{cleveref}
\usepackage{url}
\usepackage{txfonts}
\usepackage{algorithm,algorithmic}
\usepackage{epstopdf}
\usepackage{float}
\usepackage{booktabs}
\usepackage[justification=centering]{caption}
\renewcommand{\baselinestretch}{1.2}
\newtheorem{theorem}{Theorem}[section]
\newtheorem{proposition}[theorem]{Proposition}

\newtheorem{definition}[theorem]{Definition}

\newtheorem{remark}[theorem]{Remark}
\renewcommand{\thefootnote}

\begin{document}
\renewcommand{\baselinestretch}{1.2}

\title {Finding the spectral radius of a nonnegative irreducible symmetric tensor via DC programming\thanks{
This work was supported by the the National Natural Science Foundation of China [grand number 12271187]}}

\author{Xueli Bai\thanks{School of Mathematics and Statistics, Guangdong University of Foreign Studies, Guangzhou 510006, People's Republic of China.(202210047@oamail.gdufs.edu.cn).}
\and Dong-Hui Li\thanks{School of Mathematical Sciences, South China Normal University, Guangzhou 510631, People's Republic of China.(lidonghui@m.scnu.edu.cn, 2018021699@m.scnu.edu.cn).}
\and Lei Wu\thanks{School of Mathematics and Statistics, Jiangxi Normal University, Nanchang 330022, People's Republic of China.
(wulei@jxnu.edu.cn). This author was supported by National Natural Science Foundation of China [grant number 12161046].}
\and Jiefeng Xu\footnotemark[3]
}

\maketitle

\begin{abstract}
The Perron-Frobenius theorem says that the spectral radius of an irreducible nonnegative tensor is the unique positive eigenvalue corresponding to a positive eigenvector. With this in mind, the purpose of this paper is to find the spectral radius and its corresponding positive eigenvector of an irreducible nonnegative symmetric tensor. By transferring the eigenvalue problem into an equivalent problem of minimizing a concave function on a closed convex set, which is typically a DC (difference of convex functions) programming, we derive a simpler and cheaper iterative method. The proposed method is well-defined. Furthermore, we show that both sequences of the eigenvalue estimates and the eigenvector evaluations generated by the method $Q$-linearly converge to the spectral radius and its corresponding eigenvector, respectively. To accelerate the method, we introduce a line search technique. The improved method retains the same convergence property as the original version. Preliminary numerical results show that the improved method performs quite well.
\end{abstract}

{\bf Keywords.} \  Irreducible nonnegative tensor, Spectral radius, DC programming, $Q$-linear convergence.

{\bf  Mathematics Subject Classification (2020).} \ 15A18, 15A69, 90C90

\section{Introduction}
\setcounter{equation}{0}

Let $\mathbb{R}$ be the real field. For any positive integers $m$ and $n$, a real $m$-th order $n$-dimensional tensor
${\cal A}$ is an $m$-way array, which is defined as 
$$
{\cal A}=(a_{i_1i_2\ldots i_m})~\mbox{ with }~a_{i_1i_2\ldots i_m}\in\mathbb{R},~ \forall i_j\in[n], j\in[m],
$$
where $[n]:=\{1,2,\cdots,n\}$. For simplicity, the set of all real $m$-th order $n$-dimensional tensors is denoted by $\mathbb{R}^{[m,n]}$. A tensor is symmetric if its elements are invariant under any permutation of their indices. The set of all real $m$-th order $n$-dimensional symmetric tensors is given by $\mathbb{S}^{[m,n]}$. For a tensor ${\cal A}=(a_{i_1\ldots i_m})\in\mathbb{R}^{[m,n]}$ and a vector $x=(x_1,\ldots,x_n)^T\in\mathbb{R}^n$, recall that ${\cal A}x^{m-1}\in\mathbb{R}^n$ whose $i$-th element is defined as
$$
({\cal A}x^{m-1})_i:=\sum^n_{i_2,\ldots,i_m=1}a_{ii_2\ldots i_m}x_{i_2}\cdots x_{i_m},\quad \forall i\in [n],
$$
the scalar ${\cal A}x^m\in\mathbb{R}$ is defined by
$$
{\cal A}x^m:=\langle x,{\cal A}x^{m-1}\rangle=\sum^n_{i_1,\ldots,i_m=1}a_{i_1\ldots i_m}x_{i_1}\cdots x_{i_m},
$$
and ${\cal A}x^{m-2}\in\mathbb{R}^{n\times n}$, whose $(i,j)$-th element is given by
$$
({\cal A}x^{m-2})_{ij}:=\sum^n_{i_3,\ldots,i_m=1}a_{iji_3\ldots i_m}x_{i_3}\cdots x_{i_m},\quad \forall i,j\in [n].
$$
In the above definition, for any two vectors $x,y\in \mathbb R^n$, $\langle x,\,y \rangle$ is their inner produce, i.e.,
$\langle x,\,y \rangle:=\sum_{i=1}^nx_iy_i$. It is well-known that for a symmetric tensor ${\cal A}\in \mathbb{S}^{[m,n]}$, the gradient and the Hessian of the homogeneous polynomial ${\cal A}x^m$ are $m{\cal A}x^{m-1}$ and $m(m-1){\cal A}x^{m-2}$, respectively.
\begin{definition}\label{H}{\rm(\cite{Lim05,Qi05})}
Let ${\cal A}\in\mathbb{R}^{[m,n]}$. We say that $\lambda\in\mathbb{R}$ is an $H$-eigenvalue
and $x=(x_1,\ldots,x_n)^T\in\mathbb{R}^n$ is the corresponding $H$-eigenvector of ${\cal A}$ if
\begin{equation}\label{H-eigen}
{\cal A}x^{m-1}=\lambda x^{[m-1]}\quad \mbox{and}\quad \sum_{i=1}^m x_i^m=1,
\end{equation}
where $x^{[m-1]}:=(x_1^{m-1},\ldots,x_2^{m-2})^T\in\mathbb{R}^n$. And the spectral radius $\rho({\cal A})$ of
${\cal A}$ is the maximum modulus of its $H$-eigenvalues, which is given by
\begin{equation}\label{H-spe}
\rho({\cal A}):=\max\{|\lambda|\mid\lambda\mbox{ is an $H$-eigenvalue of }{\cal A}\}.
\end{equation}
\end{definition}
Now we recall three kinds of significant structured tensors.
\begin{definition}\label{stru-t}{\rm(\cite{CPZ08})}
A tensor ${\cal A}=(a_{i_1\ldots i_m})\in\mathbb{R}^{[m,n]}$ is called
\begin{description}
  \item[(i)] nonnegative, if all of its elements are nonnegative;
  \item[(ii)] reducible, if there is a nonempty proper index set $\mathbb{I} \subset [n]$ such that
      $$
      a_{ii_2\ldots i_m}=0,\quad \forall i\in I,\; \forall i_2,\ldots, i_m\not\in \mathbb{I}.
      $$
\end{description}
If ${\cal A}$ is not reducible, then, it is called irreducible.
\end{definition}

Irreducible nonnegative tensors form an important class of tensors, which possesses many nice properties. One of the most important properties is given by Chang et al. in \cite{CPZ08} by extending the well-known Perron-Frobenius theorem for the nonnegative matrix to the tensor case.
\begin{theorem} {\rm(\cite{CPZ08})}\label{PF}
If ${\cal A}\in\mathbb{R}^{[m,n]}$ is an irreducible nonnegative tensor, then, $\rho({\cal A})$ is a positive $H$-eigenvalue $\bar{\lambda}$ with a positive $H$-eigenvector $\bar{x}$. Furthermore, $\bar{\lambda}$ is the unique $H$-eigenvalue of ${\cal A}$ with a nonnegative $H$-eigenvector, and $\bar{x}$ is the unique nonnegative $H$-eigenvector associated with $\bar{\lambda}$, up to a multiplicative constant.
\end{theorem}

The tensor eigenvalue problem appears in various practical problems, such as the spectral hypergraph theory \cite{HQX15},
the higher-order Markov chains \cite{LN14}, the quantum information \cite{NQB14}, the signal processing \cite{QT03} and the magnetic resonance imaging\cite{QWW08}, to name a few.
Due to the good properties mentioned in Theorem \ref{PF}, the largest $H$-eigenvalue problem of the irreducible nonnegative tensor has received much attention both in theory(see e.g., \cite{DYB20,GTH19,QCC18,QL17,YY10,YY11,WWW20}) and algorithms. There have been developed many classes of valid methods including the first-order method, Newton's method, the homotopy continuation method, and so on(see e.g., \cite{CHHZ19,LGL16,LC19,NQ15,YN17,ZB21}).

Our purpose here is to design an efficient first-order method to find the unique largest positive $H$-eigenpair of an irreducible
nonnegative symmetric tensor. The studies on this topic have taken good progress. Lu and Pan \cite{LP16}
proposed a higher-order power method, which is a generalization of the power method for finding tensor $Z$-eigenpairs proposed by
Kolda and Mayo in \cite{KM11}. The higher-order power method is a simple and effective method. Whereas, in general,
it guarantees to neither converge nor find the largest eigenvalue of the involved tensors. Kolda and Mayo in \cite{KM14} also presented a generalized eigenproblem adaptive power method by adding and choosing automatically a shift item. Another efficient
first-order algorithm is the so-called NQZ method proposed by Ng et al. \cite{NQZ09}, which is an extension of
the method specializing for eigenvalue problems of irreducible nonnegative matrices. However, the convergence of
this algorithm is still not guaranteed. To address this problem, Chang et al. \cite{CPZ11} showed the convergence of
NQZ method for primitive tensors and Liu et al. \cite{LZI10} presented the convergence of NQZ method by exploring the
relationship between primitive tensors and irreducible nonnegative tensors. The sequence of the eigenvalue estimates generated by NQZ
method is proved  by Zhang and Qi in \cite{ZQ11} to be Q-linearly convergent for essentially positive tensors. But so far, there is no first-order
method where the generated sequence of eigenvector estimates is $Q$-linearly convergent.

In this paper, we will derive a first-order method in which both sequences of eigenvalue estimates and eigenvector estimates are $Q$-linearly convergent. Moreover, the generated sequence of eigenvalue estimates is monotonically convergent. We first reformulate the target problem into an optimization system of minimizing a concave function on a closed bounded set, which is typically a DC programming. By making a deep analysis to the equivalent system, we derive a simpler and cheaper iterative method called power-like method to find the largest positive $H$-eigenvalue and its corresponding $H$-eigenpair of the involved tensor. Without using line search, we prove that both sequences of the eigenvalue and eigenvector estimates are globally and $Q$-linearly convergent. To speed up the method, we then introduce a line search and two BB steps, which leads to the fact that the improved methods still possess $Q$-linear convergence property. A mass of numerical experiments show the high efficiency of the improved methods.

The rest of this paper is organized as follows. In \cref{formu}, we reformulate the $H$-eigenvalue problem and investigate its nice properties. In \cref{method}, we propose a simple iterative method called power-like method for finding the spectral radius of the irreducible nonnegative symmetric tensor and its corresponding positive eigenvector. And the $Q$-linear convergence of the proposed method is provided in \cref{Qconv}. In order to improve the efficiency of the method, in \cref{improv}, by the use of a line search technique, we propose a improved power-like method and show its convergence. Numerical experiments are presented in \cref{numer} to indicate the efficiency of the methods. In \cref{conc}, we give the conclusion of this paper.

\section{Reformulated problem and its properties}\label{formu}
\setcounter{equation}{0}

In this section, we transfer the $H$-eigenvalue problem (\ref{H-eigen}) into an equivalent optimization problem
and investigate its nice properties.

Suppose   ${\cal A}\in\mathbb{S}^{[m,n]}$.
It is easy to see that finding an $H$-eigenpair of $\cal A$ is equivalent to finding a KKT point of the following  optimization problem
\begin{equation}\label{H-spectral-a}
\max_{x\in {\mathbb B}} \phi(x):={\cal A}x^m,
\end{equation}
where  ${\mathbb  B} =\{x\in \mathbb R^n\; |\; \sum_{i=1}^n x_i^m=1\}$ . Moreover, finding the largest $H$-eigenvalue of
$\cal A$ and its corresponding eigenvector is equivalent to finding the global solution of the above problem, which is
generally NP-hard. However, when $\cal A$ is irreducible and nonnegative, the problem becomes a lot easier. In this case,
the largest eigenvalue $\lambda^*$ together with its corresponding eigenvector $x^*$, is the unique solution to the optimization problem
\begin{equation}\label{H-spectral}
\max _{x\in {\mathbb  B}_+} \phi(x)= {\cal A}x^m
\end{equation}
with $\lambda^*={\cal A}(x^*)^m$, where
$$
{\mathbb  B}_+=\Big \{x\in \mathbb R_{++}^n\;\Big  |\;\sum_{i=1}^n x_i^m=1\Big \},
\;
\mathbb{R}^n_{++}:=\Big \{x\in\mathbb{R}^n\; \Big |\; x_i>0, \forall  i\in[n]\Big \}.
$$
Denote the Lagrangian function of the system (\ref{H-spectral-a}) as
\begin{equation}\label{def:Lx}
L(x,\lambda)=-{\cal A}x^m + \lambda \Big(\sum_{i=1}^m x_i^m-1\Big).
\end{equation}
The KKT condition of (\ref{H-spectral-a}) is
$$
\nabla _x L(x,\lambda)=0\quad\mbox{and}\quad \sum_{i=1}^n x_i^m=1.
$$
Obviously, $(\bar x,\bar \lambda)$ is a KKT point of the system (\ref{H-spectral-a}) if and only if it is an eigenpair of $\cal A$.
If $\cal A$ is irreducible and nonnegative, then a KKT point $(\bar x,\bar \lambda)$ of (\ref{H-spectral-a}) with $\bar x\in \mathbb R^n_{++}$ is the unique $H$-eigenpair of ${\cal A}$. What is more, it is also a KKT point of (\ref{H-spectral}).

In the remainder of this paper, we always suppose that ${\cal A}\in\mathbb{S}^{[m,n]}$ is irreducible nonnegative. We first rewrite the homogeneous polynomial function $\phi(x)$ in the form
$$
\phi(x)= {\cal A}x^m=\sum_{i=1}^K c _i  x_1^{a_{1i}}x_2^{a_{2i}}\cdots x_n^{a_{ni}},
$$
where $K$ is the number of the different terms in the summation and for each $i\in [K]$, $c_i> 0$
is the summation of all the coefficients in ${\cal A}x^m$ corresponding to the term $x_1^{a_{1i}}x_2^{a_{2i}}\cdots x_n^{a_{ni}}$.
Denote, for each $i\in[K]$, the nonnegative integer vector
$a_i=(a_{1i}, a_{2i}\ldots, a_{ni})^T\in\mathbb R^n$. 
The change of variable $y_i=\log  x_i$ with $x_i>0$ turns $ c_ix_1^{a_{1i}}x_2^{a_{2i}}\cdots x_n^{a_{ni}}$, the monomial function, into an exponential of an affine function, which results in
$$
\phi(e^{[y]})= \sum_{i=1 } ^K   e^{a_i^Ty+b_i},
$$
where for any $i\in[K]$, $b_i= \log {c_i}$ and $e^{[y]}=(e^{y_1},\ldots, e^{y_n})^T\in\mathbb  R^n$. Thus, we get the following equivalent system to problem (\ref{H-spectral}):
\begin{equation}\label{H-opt}\left \{ \begin{array}{rll}
\min & f(y):= -\log(\phi(e^{[y]}))= -\log \Big(\displaystyle \sum\limits_{i=1 } ^K   e^{a_i^Ty+b_i} \Big), \\
\mbox{s.t.} & \displaystyle \sum_{i=1}^n e^{my_i}-1=0.
\end{array}\right.
\end{equation}
Obviously, systems (\ref{H-opt}) and (\ref{H-eigen}) are completely equivalent in the sense that a positive vector $x$ is a solution to (\ref{H-eigen})
if and only if $y=\log x$ is a solution to (\ref{H-opt}). Throughout, for a positive vector $x$,
the meaning of $\log x$ is done by elements. From Perron-Frobenius theorem, we know that system (\ref{H-eigen})
has a unique positive solution $\bar x$. As a result, the optimization system (\ref{H-opt}) has a unique solution $\bar y=\log \bar x$.

It is clear that  $f(y)=-\log \phi(x(y))$ with $x(y)=e^{[y]}$. 
Then the gradient of this composite function $f(y)$ is given by
\begin{equation}\label{gradient}
\nabla f(y)=-\frac {\mbox{diag}\Big(e^{[y]} \Big) \nabla \phi(x(y))}{\phi(x(y))}
= -m\frac {{\cal A}x^{m-1}\circ x}{{\cal A}x^m},
\end{equation}
where ``$\circ $" denotes the Hardarmard product of two vectors. 
It is interesting to note that for any $y\in \mathbb R^n$,
\begin{equation}\label{const}
\sum_{i=1}^n \frac {\partial f(y)}{\partial y_i} =-m.
\end{equation}
The Lagrangian function of problem (\ref{H-opt}) is denoted by
$$
\bar L(y,\mu)= f(y) + \frac 1m  \mu \Big(\sum_{i=1}^n e^{my_i}-1 \Big),
$$
and the KKT condition is
$$\left \{\begin{array}{l}
\nabla f(y)+  \mu e^{[my]}=0,\\
\displaystyle\sum_{i=1}^n e^{my_i}-1=0.
\end{array}\right.
$$
This together with (\ref{const}) implies $\mu = m$, which is independent of $x$.
As a result, we claim that the solution $\bar y$ of (\ref{H-opt}) satisfies
\begin{equation}\label{kkt-H}
\nabla f(\bar y)+ me^{[m\bar y]}=0\quad\mbox{and}\quad \sum_{i=1}^n e^{m\bar y_i}-1=0,
\end{equation}
and the corresponding Lagrangian multiplier $\mu$ equals to $m$.
It is clear that $\bar y$ is also a solution of the problem 
\begin{equation}\label{def:Ly}
  \min_{y\in \mathbb{R}^{n}}L(y):=\bar L(y,m) = f(y)+ \Big(\sum_{i=1}^n e^{my_i}-1 \Big).  
\end{equation}

By letting $\bar x=e^{[\bar y]}$, we can obtain from (\ref{gradient}) and (\ref{kkt-H}) that
$$
0= -m\frac {{\cal A}\bar{x}^{m-1}\circ \bar{x}}{{\cal A}\bar{x}^m} +m \bar{x}^{[m]},
$$
which shows that $\bar x>0$ is an H-eigenvector of $\cal A$ corresponding to the eigenvalue $\bar\lambda =  {\cal A}\bar x^m>0$. Now, we summarize the above arguments in the following theorem.

\begin{theorem}\label{th:kkt-H}
If $\bar y$ is a KKT point of the optimization problem (\ref{H-opt}), then, $\bar x=e^{[\bar y]}$ is a positive eigenvector of $\cal A$ corresponding to the spectral radius $\bar\lambda = {\cal A}\bar x^m$. Conversely, if $(\bar x,\bar\lambda)$ is the unique positive eigenpair of $\cal A$, then, $\bar y=\log \bar x$ is the unique solution of (\ref{H-opt}), and the corresponding Lagrangian multiplier $\mu$ equals to $m$, the order of ${\cal A}$.
\end{theorem}

According to Perron-Frobenius theorem and Theorem \ref{th:kkt-H}, the local solution of the optimization problem (\ref{H-opt}) and the KKT point are both unique. Consequently, $\bar y$ is a solution of the problem (\ref{H-opt}) if and only if it satisfies the KKT condition (\ref{kkt-H}).

Next, we derive the second order properties of problem (\ref{H-opt}).
First, it is easy to obtain the following proposition.
\begin{proposition}\label{prop:psd}
	For any $y\in\mathbb R^n$, the matrix $-\nabla ^2f(y)$ is symmetric and positive semidefinite.
\end{proposition}
\begin{proof}
Note that $-f(y)$ is the composite of the log-sum function
$$g(z) = \log \Big( \sum_{i=1 } ^K   e^{z_i} \Big)$$
with an affine transformation $z_i = {a_i^Ty+b_i}, i\in [K]$.
Since the log-sum function $g(z)$ is convex, we have $f(y)$ is concave by the convexity of composite function.
\end{proof}
On the other hand, by direct calculation, we obtain that
\begin{eqnarray}\label{Hession}
\nabla ^2 f(y) &=& \frac {m^2} {\phi(x)^2} \mbox{diag}(x) \Big({\cal A}x^{m-1} \Big)\Big({\cal A}x^{m-1} \Big)^T \mbox{diag}(x)\nonumber \\
    &&    - \frac {m(m-1)}{\phi(x)}\mbox{diag}(x) {\cal A}x^{m-2}  \mbox{diag}(x)
        - \frac {m}{\phi(x)}\mbox{diag}\Big(x  \circ {\cal A}x^{m-1} \Big)\nonumber\\
        &=& \frac {m} {\phi(x)} \mbox{diag}(x) B(x)\mbox{diag}(x)
\end{eqnarray}
where $x=e^{[y]}$ and
\begin{equation}\label{def:Bx}
	B(x)= \frac {m}{\phi(x)} \Big({\cal A}x^{m-1} \Big)\Big({\cal A}x^{m-1} \Big)^T
	-(m-1) {\cal A}x^{m-2}  -\mbox{diag}\Big(\frac {{\cal A}x^{m-1}}{x} \Big)
\end{equation}
with $\frac {{\cal A}x^{m-1}}{x} $ being a vector whose elements are $({\cal A}x^{m-1})_i\,/ x_i $ for all $i\in[n]$.

Furthermore, suppose that $(\bar y,\bar\mu)$ is a KKT point of the system (\ref{H-opt}) with $\bar \mu=m$. Let $\bar x=e^{[\bar y]}$
and $\bar \lambda ={\cal A}\bar x^m=\phi(\bar x)$. Then, ${\cal A}\bar x^{m-1}=\bar\lambda\bar x^{[m-1]}$ and
\begin{eqnarray}\label{bar:B}
B(\bar x) &=& m\bar\lambda \bar x^{[m-1]} \Big(\bar x ^{[m-1]}\Big)^ T
        -(m-1) {\cal A}\bar x^{m-2}  - \bar \lambda \mbox{diag}\left(\bar x^{[m-2]} \right)\nonumber \\
        &=& m\bar\lambda \bar x^{[m-1]} \Big(\bar x ^{[m-1]}\Big)^ T+\frac{1}{m}\nabla _x^2L(\bar x,\bar\lambda) -m\bar\lambda  \mbox{diag}\left(\bar x^{[m-2]} \right).
\end{eqnarray}
Hence, we can obtain
\begin{eqnarray}\label{Hess:x-y}
\nabla ^2L(\bar y) &=&
=\nabla ^2f(\bar y) + m ^2 \mbox{diag}\Big(\bar x^{[m]} \Big) \nonumber \\
    &=& m\bar\lambda ^{-1} \mbox{diag}(\bar x) \Big [ B(\bar x) + m \bar\lambda \mbox{diag}\left(\bar x^{[m-2]} \right)\Big ]
    \mbox{diag}(\bar x) \nonumber \\
    &=& \bar\lambda ^{-1} \mbox{diag}(\bar x) \Big [ m^2 \bar\lambda \bar x^{[m-1]} \Big(\bar x ^{[m-1]}\Big)^ T
        + \nabla _x^2 L(\bar x,\bar\lambda) \Big ]    \mbox{diag}(\bar x).
\end{eqnarray}

Now, we end this section with the following theorem.
\begin{theorem}\label{th:suff}
Suppose that ${\cal A}\in\mathbb{S}^{[m,n]}$ is irreducible and nonnegative. Let $\bar x\in\mathbb {R}^n_{++}$
and $\bar y=\log \bar x$ be solutions of systems (\ref{H-spectral}) and (\ref{H-opt}), respectively. Then, the following statements are true.
\begin{description}
\item [(i)] A KKT point $\bar x\in\mathbb {R}^n_{++}$ satisfies the second order sufficient condition of
system (\ref{H-spectral}) if and only if $\bar y=\log \bar x$ satisfies the second condition of system (\ref{H-opt}).
  \item [(ii)] At the positive eigenpair $(\bar x,\bar\lambda)$, the second order sufficient condition of system (\ref{H-spectral}) holds.
  \item [(iii)]  At the positive eigenpair $( \bar x,\bar\lambda)$, it holds that
  $$
   \nabla _x^2L(\bar x,\bar\lambda)d\neq 0,\quad \forall d\neq 0: \; \Big(\bar x^{[m-1]} \Big)^Td=0,
  $$
  and
  $$
   \nabla^2 L(\bar y)p\neq 0,\quad \forall p\neq 0: \; \Big(e^{[m\bar y]} \Big)^Tp=0.
  $$
  As a result, the matrices
  $$
   \bar P= \left( \begin{array}{c} \nabla _x^2L(\bar x,\bar\lambda) \\  \Big(\bar x^{[m-1]} \Big)^T \end{array}\right)
   \quad \mbox{and}\quad
   \bar Q= \left( \begin{array}{c} \nabla ^2L(\bar y) \\  \Big(e^{[m\bar y]} \Big)^T \end{array}\right)
  $$
  are of full rank. Here, $L(x,\lambda)$ and $L(y)$ are defined by (\ref{def:Lx}) and (\ref{def:Ly}), respectively.
\end{description}
\end{theorem}
\begin{proof}
(i) Suppose that $(\bar y,\bar\mu)$ is a KKT point of system (\ref{H-opt}) with $\bar \mu=m$. The second order sufficient condition of system (\ref{H-opt}) at $(\bar y,\bar\mu)$ means that
$$
d^T\nabla ^2L(\bar y)d >0,\quad \forall d:\; \Big(e^{[m\bar y]}\Big)^Td=0.
$$
According to (\ref{Hess:x-y}), the expression of $\nabla ^2L(\bar y)$, it is easy to show that the last condition is equivalent to
$$
d ^T\Big [ m^2 \bar\lambda \bar x^{[m-1]} \Big(\bar x ^{[m-1]}\Big)^ T + \nabla _x^2 L(\bar x,\bar\lambda)\Big ] d>0, \quad
\forall d:
\Big(\bar x^{[m-1]} \Big)^T d=0,
$$
or
$$
d ^T\nabla _x^2 L(\bar x,\bar\lambda) d>0, \quad
\forall d:
\Big(\bar x^{[m-1]} \Big)^T d=0,
$$
which is equivalent to saying that $(\bar x,\bar\lambda)$ satisfies the second order sufficient condition of problem (\ref{H-spectral}).

(ii) Denote
\begin{equation}\label{bar A}
    \bar A= \mbox{diag}\Big(\bar{x}^{[-\frac {m-2}{2}]} \Big){\cal A}\bar x^{m-2} \mbox{diag}\Big(\bar{x}^{[-\frac {m-2}{2}]} \Big)\quad
\mbox{and}\quad \bar z=  \bar x ^{[\frac m2]}.
\end{equation}
The condition that $(\bar x,\bar\lambda)$ is an eigenpair of $\cal A$ implies
$$
\bar A \bar z=\Big [ \mbox{diag}\Big(\bar{x}^{[-\frac {m-2}{2}]} \Big)
 {\cal A}\bar x^{m-2} \mbox{diag}\Big(\bar{x}^{[-\frac {m-2}{2}]} \Big) \Big ]  \bar x ^{[\frac m2]} =\bar \lambda  \bar x ^{[\frac m2]}=\bar\lambda \bar z.
$$
In other words, $(\bar z,\bar\lambda)$ is a positive eigenpair of the irreducible nonnegative symmetric matrix $\bar A$, and then, $\bar \lambda$ is the largest eigenvalue of $\bar A$. Consequently, $\bar\lambda $ satisfies
\begin{equation}\label{tenp:th2.3}
\bar\lambda =\max_{p\neq 0}\frac {p^T\bar Ap}{\|p\|^2}.
\end{equation}
Since the positive eigenvector $\bar x$ is the solution of the optimization problem (\ref{H-spectral}), the second order necessary
condition holds. Thus, $(\bar x,\bar\lambda)$ satisfies
$$
d^T\nabla ^2_x L(\bar x,\bar \lambda)d\ge 0,\quad\forall d:\;  \Big(\bar x^{[m-1]}\Big)^Td=0.
$$
It suffices to show the strict inequality holds for all $d\neq 0$ satisfying $\Big(\bar x^{[m-1]}\Big)^Td=0$.

Suppose on the contrary that there is some $d\neq 0$ satisfying $\Big(\bar x^{[m-1]}\Big)^Td=0$ but
$d^T\nabla ^2_x L(\bar x,\bar \lambda)d=0$, i.e., 
$d^T {\cal A} \bar x^{m-2}d = \bar\lambda d^T \mbox{diag}\left(\bar x^{[m-2]}\right)d=0.$
Denote $p=\bar{x}^{[\frac {m-2}{2}]}\circ d$, then, we can get from (\ref{bar A}) that
$$
\frac {p^T\bar Ap}{\|p\|^2}=
\frac {d^T
 {\cal A}\bar x^{m-2}d}
 { d^T \mbox{diag}\left(\bar x^{[m-2]}\right)d}=\bar \lambda.
$$
This indicates that $p$ is a solution of (\ref{tenp:th2.3}). Consequently, $p$ is an eigenvector of $\bar A$ corresponding to the eigenvalue $\bar\lambda$. Since $(\bar{z},\bar\lambda)$ is a positive eigenpair of $\bar A$, according to Perron-Frobenius theorem, there must exist a scalar $\beta$ such that $p=\beta \bar z=\beta \bar x ^{[\frac m2]}$. However, $d$ satisfies $\Big(\bar x^{[m-1]}\Big)^Td=0$. Therefore,
$$
0=  \Big(\bar x^{[m-1]}\Big)^Td=\Big(\bar x^{[m-1]}\Big)^T\Big(\bar{x}^{[-\frac {m-2}{2}]}\circ p \Big)
    = \beta \| \bar{x}^{[\frac {m}{2}]} \|^2\neq 0,
$$
which yields a contradiction. Thus, the second order sufficient condition of problem (\ref{H-spectral}) holds at $(\bar x,\bar \lambda)$.

(iii) For any $d$ satisfying $\Big(\bar x^{[m-1]} \Big)^T d=0$, the equality $ \nabla _x^2 L(\bar x,\bar\lambda) d=0$ implies
$$
\Big(\mbox{diag}(\bar x^{[2-m]}) {\cal A}\bar x^{m-2} - \bar\lambda I\Big) d= 0.
$$
On the other hand, it obviously holds that
$$
\Big(\mbox{diag}(\bar x^{[2-m]}) {\cal A}\bar x^{m-2} - \bar\lambda I \Big) \bar x=  \mbox{diag}(\bar x^{[2-m]}) {\cal A}\bar x^{m-1} - \bar\lambda \bar x = 0.$$
If $d\neq 0$, then, both $\bar x$ and $d$ are  eigenvectors of the irreducible nonnegative matrix $\mbox{diag}(\bar x^{[-(m-2)]}) {\cal A}\bar x^{m-2}$ corresponding to the positive eigenvalue $\bar\lambda$. By Perron-Frobenious theorem, there is a scalar $\gamma\neq 0$ such that $d=\gamma\bar x$, which contradicts to the fact that $\Big(\bar x^{[m-1]} \Big)^T d=0$.

The conclusion in item (iii) about $\bar y$ can be proved in a similar way.
\end{proof}

\section{Power-like method and its convergence}\label{method}
\setcounter{equation}{0}

In this section, we will derive a simple but efficient iterative method for finding the spectral radius of an irreducible nonnegative symmetric tensor by solving a sequence of linearized approximation problems to system (\ref{H-opt}).

Since  function $f(y)$ is concave, the problem (\ref{H-opt}) is actually a DC programming. Hence, it can be solved by linearized method. Specifically, at each iteration, we solve a constrained optimization problem by approximating the objective function with a linear function while keeping the constraint unchanged.
Details are given below.

At the $k$-th iteration, suppose that a feasible iteration point $y_k$ has already been obtained. 
Consider the following linearized sub-problem
\begin{equation}\label{sub-H}\left \{ \begin{array}{rl}
\min & \nabla f(y_k)^T(y-y_k)\\
\mbox{s.t.} & \displaystyle \sum_{i=1}^n e^{my_i}-1=0.
\end{array}\right.
\end{equation}
The Lagrangian function of the above sub-problem is
$$L_k(y, \mu)= \nabla f(y_k)^T(y-y_k) + \frac 1m \mu  \Big(\sum_{i=1}^n e^{my_i}-1 \Big),$$
and the KKT condition is
\begin{equation}\label{kkt-sub}
    \nabla f(y_k) + \mu e^{[my]}=0 \quad\mbox{and}\quad \sum_{i=1}^n e^{m(y)_i}-1=0.
\end{equation}
The last equation and (\ref{const}) deduce $\mu =m.$
Not difficult to see that (\ref{kkt-sub}) has a unique solution $(y_{k+1}, \mu_{k+1})$ with $\mu_{k+1} = m$, i.e., sub-problem (\ref{sub-H}) has a unique solution $y_{k+1}$ satisfying
\begin{equation}\label{KKT-H}
\nabla f(y_k) +m e^{[my_{k+1}]}=0\quad\mbox{and}\quad \sum_{i=1}^n e^{m(y_{k+1})_i}-1=0.
\end{equation}
The following proposition especially implies that the sequence $\{f(y_k)\}$ is decreasing. 
\begin{proposition}\label{prof:desc}
The sequence $\{y_k\}$ generated by (\ref{sub-H}) has the following properties.
\begin{description}
  \item [(i)]  For any $k\ge 0$, the Lagrangian multiplier $\mu _{k+1}$ equals to $m$.
  \item [(ii)] The sequence of function evaluations $\{f(y_k)\}$ is non-increasing, i.e., $f(y_{k+1}) \le f(y_k)$ for all $k\ge 0$. Moreover, the equality $f(y_{k+1})=f(y_k)$ holds if and only if $y_k$ is a KKT point of the problem (\ref{H-opt}),
      and hence $x_k=e^{[y_k]}$ is a positive eigenvector of $\cal A$ corresponding to the positive eigenvalue $\lambda _k={\cal A}x_k^m$.
\end{description}
\end{proposition}
\begin{proof}
From the analysis before \cref{prof:desc}, we can directly obtain item (i). 
Now we turn to prove item (ii). 
Since $y_k$ is  feasible for the problem (\ref{sub-H}), it is clear that $\nabla f(y_k)^T(y_{k+1}-y_k)\le 0$.
By the concavity of the function $f(y)$,  we have for any $k\ge 0$,
$$
f(y_{k+1})-f(y_k)\le \nabla f(y_k)^T(y_{k+1}-y_k)\le 0.
$$
Clearly, the equality $f(y_{k+1})=f(y_k)$ holds if and only if $\nabla f(y_k)^T(y_{k+1}-y_k)= 0$,
which is equivalent to saying that $y_k$ is also a solution to the problem (\ref{sub-H}). Thus, $y_k$
together with $\mu _{k+1}=m$ satisfies the KKT condition, i.e.,
$$
\nabla f(y_k) + me^{[my_k]}=0  \quad\mbox{and}\quad \sum_{i=1}^n e^{m(y_k)_i}-1=0,
$$
which means that $y_k$ is a KKT point of the problem (\ref{H-opt}) with $\mu _k=m$ being the Lagrangian multiplier. 
The rest results about $x_k$ directly follows Theorem \ref{th:kkt-H}.
\end{proof}

In what follows, we derive the iterative scheme of the system (\ref{sub-H}) with variable $x$. Let $x_k=e^{[y_k]}$. It follows from (\ref{KKT-H}) and (\ref{gradient}) that
$$
-m \frac {{\cal A}x_k^{m-1}\circ x_k}{{\cal A}x_k^m} + m x_{k+1}^{[m]}=0.
$$
Denote
$$
F(x):=-{\cal A}x^{m-1}+{\cal A}x^m\cdot x^{[m-1]}.
$$
It is easy to see that $F(x)=0$ if and only if $x$ is an $H$-eigenvalue of ${\cal A}$ with
corresponding eigenvector $\lambda={\cal A}x^m$. As a result, we obtain an explicit iterative scheme, which is shown in Algorithm \ref{alg1} below.

\begin{algorithm}[!htbp]
\caption{(Power-like method for finding the spectral radius)}\label{alg1}
\begin{algorithmic}[1]
\STATE Given an initial point $x_0\in \mathbb{B}_{+}$. Let $k:=0$ and $\lambda_0={\cal A}x_0^m$.
\WHILE{$F(x_k)\neq 0$}
\STATE Compute $x_{k+1}$ and $\lambda_{k+1}$ by
\begin{equation}\label{iter-H}
x_{k+1}= \left(\frac {{\cal A}x_k^{m-1}\circ x_k}{{\cal A}x_k^m}\right)^{[\frac{1}{m}]}\quad\mbox{and}\quad\lambda_{k+1}={\cal A}x_{k+1}^m.
\end{equation}
Let $k:=k+1$.
\ENDWHILE
\end{algorithmic}
\end{algorithm}

Iterative scheme (\ref{iter-H}) of Algorithm \ref{alg1} looks very like but different from higher-order power method. Consequently, we call it power-like method.

\begin{theorem}\label{th:positive-H}
Let the sequences $\{x_k\}$ and $\{\lambda _k\}$ be generated by Algorithm \ref{alg1} with a positive feasible initial point $x_0$, and $\{y_k\}$ be the corresponding sequence satisfying $x_k=e^{[y_k]}$. Then,
\begin{description}
  \item [(i)] the positive sequence $\{\lambda_k\}$ is increasing and converges to some $\bar \lambda>0$; and
  \item [(ii)] the sequence $\{y_{k}\}$ is bounded and $\{x_k\}$ has a positive lower bound.
\end{description}
\end{theorem}
\begin{proof}(i) Obviously, the assumption that $\cal A$ is irreducible and nonnegative leads to the fact that both the sequences $\{x_k\}$ and $\{\lambda_k\}$ are positive.

From Proposition \ref{prof:desc} we know that the sequence $\{f(y_k)\}$ is decreasing. This implies that the sequence $\{-\log \phi(x_k)\}$ is decreasing. Thus, the sequence $\{\phi(x_k)\}$ is increasing, which is equivalent to saying that the sequence $\{\lambda_k\}$ is increasing. Since the sequence $\{x_k\}\subset \mathbb B_+$ is bounded, the sequence $\{\lambda_k\}$ is bounded, too. As a result, we claim that the sequence $\{\lambda_k\}$ converges to some $\bar\lambda>0$.

(ii) Actually, we only need to show the sequence $\{x_k\}$ has a positive lower bound.

Since $\{x_k\}\subset \mathbb B_+$ is bounded and $\cal A$ is nonnegative, the sequence $\{{\cal A}x_k^{m-1}\}$ is also bounded, and
$$
\bar \epsilon :{=} \mathop {\lim\inf}_{k\to\infty} {\cal A}x_k^{m-1}\ge 0.
$$
Let $\rho =\frac 1m$, $\eta _i= \frac {\bar\epsilon _i} {\bar \lambda }$ for any $i\in [n]$, and define two index sets
$$
{\mathbb J}=\{i\;|\; \bar\epsilon_i>0\}\quad\mbox{and}\quad
\mathbb I=\{i\;|\; \bar\epsilon_i=0\}.
$$
We prove item (ii) by showing that the subvector $(x_{k+1})_{\mathbb J}$ in bounded away from $0$ and $\mathbb I=\emptyset$.

Notice that ${\cal A}x_k^m=\lambda_k\le \bar\lambda$.
It follows from (\ref{iter-H}) that for any $i\in \mathbb J$ and $k\ge 0$,
\begin{eqnarray*}
(x_{k+1})_i &=&  \Big(\frac {{\cal A}x_k^{m-1}}{\lambda _k} \Big)_i ^{\frac 1m }\Big (x_k\Big)_i^{\frac 1m} \\
&\ge&  \Big(\frac {\bar\epsilon _i} {\bar \lambda } \Big) ^{\frac 1m}\Big(x_k\Big )_i^{\frac 1m}
=\eta _i  {} ^{\rho} \Big(x_k\Big )_i^{\rho }\\
&\ge&  \cdots \ge \eta _i {}^{\rho +\rho ^2 +\cdots +\rho ^{k+1}} \Big(x_0\Big)_i^{\rho ^{k+1}}.
\end{eqnarray*}
The last inequality implies that for each $i\in \mathbb J$, $\{(x_k)_i\}$ has a positive lower bound.

Next, we verify $\mathbb I=\emptyset$. Suppose on the contrary that $\mathbb I\neq \emptyset$. Then, there exsits an infinite set $K$ such that for each $i\in \mathbb I$, we have
$$
\lim _{k\in K,\,k\to \infty} \Big({\cal A}x_k^{m-1}\Big)_i=0.
$$
Since $\{x_k\}\subset \mathbb B_+$ is bounded, there is a subsequence of $\{x_k\}_{k\in K}$ converging to some $\bar x\in\mathbb  B_+$. Without loss of generality, suppose that $\{x_k\}_{k\in K}\to \bar x$ as $k\rightarrow\infty$. Then, it follows from the last equality that for any $i\in \mathbb I$,
$$
\sum_{i_2,\ldots, i_m} a_{i i_2 \ldots i_m}\bar x_{i_2}\cdots \bar x_{i_m}=0.
$$
Since we have already shown $\bar x_i>0$ for any $i\in\mathbb  J$, the last equality indicates that
$$
a_{i i_2 \ldots i_m}=0,\quad \forall i\in\mathbb  I\quad\mbox{and}\quad \forall i_2,\ldots, i_m\not\in \mathbb I.
$$
This contradicts to the fact that $\cal A$ is irreducible. Therefore, $\mathbb I=\emptyset$ and the item (ii) is true.
\end{proof}

The convergence of the power-like method is stated in the following theorem.
\begin{theorem}\label{th:conv-H}
Let the sequences $\{x_k\}$ and $\{\lambda _k\}$ be generated by Algorithm \ref{alg1}, and $\{y_k\}$ be the corresponding sequence satisfying $x_k=e^{[y_k]}$. Then, the limit of the sequence $\{\lambda_k\}$ is the unique positive H-eigenvalue $\bar \lambda$ of $\cal A$. Moreover, the sequence $\{x_k\}$ converges to a positive eigenvector $\bar x$ of $\cal A$ corresponding to $\bar \lambda$.
\end{theorem}
\begin{proof}
By Perron-Frobenius theorem, it suffices to  show that the limit $\bar\lambda$ of $\{\lambda_k\}$ is the unique positive $H$-eigenvalue of $\cal A$ and every accumulation point of the sequence $\{x_k\}$ is a positive eigenvector of $\cal A$ corresponding to the eigenvalue $\bar \lambda$.

Theorem \ref{th:positive-H} has shown that every accumulation point of the sequence $\{x_k\}$ is positive. Suppose that for an infinite set $K$, the subsequence $\{x_k\}_K$ converges to some $\bar x>0$ and the subsequence $\{x_{k+1}\}_K$ converges to some $\tilde x>0$. Accordingly, the subsequences $\{y_k\}_K$ and $\{y_{k+1}\}_K$ converge to some $\bar y=\log \bar x$ and
$\tilde y=\log \tilde x$.

Since $y_{k+1}$ is a solution of the problem (\ref{sub-H}), we have that
$$
f(y_{k+1})-f(y_k)  \le   \nabla f(y_k)^T(y_{k+1}-y_k)\le 0 .
$$
Taking limits in the last inequalities as $k\to\infty$ with $k\in K$, we get $\nabla f(\bar y)^T(\tilde y-\bar y)=0$, or equivalently, $\nabla f(\bar y)^T \tilde y=\nabla f(\bar y)^T\bar y$. On the other hand, by taking limits in both equalities of system (\ref{KKT-H}) as $k\to\infty$ with $k\in K$, we can obtain
$$
\nabla f(\bar y) + m e^{[m\tilde y]}=0\quad\mbox{and}\quad \sum _{i=1}^n e^{m\tilde y_i}=1,
$$
which implies the fact that $\tilde y$, together with the multiplier $\mu =m$, is a KKT point of the convex programming problem
$$\left\{\begin{array}{lll}
\min & \nabla f(\bar y)^T(y-\bar y),  \\
\mbox{s.t.} & \displaystyle \sum _{i=1}^n e^{m y_i}=1.
\end{array}\right.
$$
Consequently, $\tilde y$ is a global solution of the above problem. Besides, the fact $\nabla f(\bar y)^T \tilde y=\nabla f(\bar y)^T\bar y$ and $\sum _{i=1}^n e^{m\bar y_i}=1$ shows that $\bar y$ is also a solution of the last problem, and hence, it together with $\mu=m$ is a KKT point of the last problem. In other words, $\bar y$ satisfies
$$
\nabla f(\bar y) + m e^{[m\bar y]}=0\quad\mbox{and}\quad \sum _{i=1}^n e^{m\bar y_i}=1,
$$
which demonstrates that $\bar y$ is a KKT point of the problem (\ref{H-opt}) or (\ref{H-opt}). By Theorem \ref{th:kkt-H}, $\bar x=e^{[\bar y]}$ is a positive eigenvector of $\cal A$ corresponding to the positive eigenvalue $\bar\lambda ={\cal A}\bar x^m$.
\end{proof}

\section{Q-linear convergence}\label{Qconv}
\setcounter{equation}{0}

In this section, we will show that the convergence rate of both sequences $\{x_k\}$ and $\{\lambda_k\}$ generated by
the power-like method are $Q$-linear, while the convergence rate of the sequence $\{y_k\}$ with $x_k=e^{[y_k]}$ is $R$-linear.

Notice that the sequence $\{y_k\}$ satisfies
$$
 \nabla f(y_k) +m e^{[my_{k+1}]}=0\quad\mbox{and}\quad
 \nabla f(\bar y) + m e ^{[m\bar y]}=0,
$$
which implies
\begin{eqnarray*}
\nabla L(y_k)&=&\nabla f(y_k)+ m e^{[my_k]}= - m\Big(e^{[my_{k+1}]}-e^{[my_k]}\Big)\\
      &=& -m^2\Big(\int _0^1 \mbox{diag}\Big(e^{[m(y_k +\tau(y_{k+1}-y_k))]} \Big) d\,\tau \Big)(y_{k+1}-y_k),
\end{eqnarray*}
and
\begin{eqnarray*}
&&\nabla L(y_k)^T(y_{k+1}-y_k)\\ &=& -m^2(y_{k+1}-y_k)^T\Big(\int _0^1 \mbox{diag}\Big(e^{[m(y_k +\tau(y_{k+1}-y_k))]} \Big) d\tau \Big)(y_{k+1}-y_k).
\end{eqnarray*}
The last equality together with Theorem \ref{th:positive-H} indicates that there are two constants $\gamma_1\ge\beta_1>0$ such that
\begin{equation}\label{est:grad-L}
\beta_1\| y_{k+1}-y_k\|\le \|\nabla L(y_k)\|\le \gamma_1 \|y_{k+1}-y_k\|,\quad \forall k\ge 0
\end{equation}
and
$$
\nabla L(y_k)^T(y_{k+1}-y_k) \le  -\gamma _1 \|y_{k+1}-y_k\|^2, \quad \forall k\ge 0.
$$

Since for any $k\ge 0$, $y_k$ satisfies $\displaystyle \sum_{i=1}^n e^{m(y_k)_i}=1$, we have $f(y_k)=L(y_k)$. By Taylor's expansion, there is a $\tilde y_k=y_k+\tilde{\theta}_k(y_{k+1}-y_k)$ with $\tilde{\theta} _k\in(0,1)$ such that
\begin{eqnarray}\label{Taylor-1}
f(y_{k+1}) &=&  L(y_{k+1}) \nonumber \\
  &=& L(y_k) +\nabla L(y_k)^T(y_{k+1}-y_k) +\frac 12(y_{k+1}-y_k)^T\nabla ^2L(\tilde y_k)(y_{k+1}-y_k)\nonumber \\
    &=& f(y_k) +\nabla L(y_k)^T(y_{k+1}-y_k)\nonumber \\
    &&+\frac 12(y_{k+1}-y_k)^T\Big [ \nabla ^2 f(\tilde y_k) + m^2 \mbox{diag}\Big(e ^{[m\tilde y_k]}\Big)\Big ](y_{k+1}-y_k)\nonumber\\
    &\le & f(y_k) +\nabla L(y_k)^T(y_{k+1}-y_k)
    +\frac 12 m^2(y_{k+1}-y_k)^T\mbox{diag}\Big(e ^{[m\tilde y_k]}\Big)(y_{k+1}-y_k)\nonumber \\
    &=&  f(y_k)-  m^2(y_{k+1}-y_k)^T\Big [\int _0^1 \mbox{diag}\Big(e^{[m(y_k +\tau(y_{k+1}-y_k))]} \Big) d\,\tau\nonumber \\
    &&-\frac 12 \mbox{diag}\Big(e ^{[m\tilde y_k]}\Big)\Big ](y_{k+1}-y_k),
\end{eqnarray}
where the inequality holds because $\nabla ^2f(y)$ is negative semi-definite for all $y$.

Since $\{y_k\}\to\bar y$ as $k\to\infty$, we claim by (\ref{Taylor-1}) that there is a constant $\beta _2>0$ such that
the inequality
\begin{equation}\label{linear-1}
f(y_{k+1}) -f(\bar y) \le f(y_k)-f(\bar y) - \beta _2\|y_{k+1}-y_k\|^2
\end{equation}
holds for all $k$ sufficiently large. Notice that $\nabla L(\bar y)=0$, and that both $y_k$ and $\bar y$ are feasible. Denote
$$
P_k=\left( \begin{array}{c} \displaystyle \int _0^1 \nabla ^2 L\Big(\bar y+\tau(y_k-\bar y)\Big)d\, \tau  \\
    \displaystyle \int _0^1 \Big(e^{[m(\bar y +\tau(y_k-\bar y))]}\Big)^T d\, \tau \end{array}\right).
$$
It is obvious that $P_k\to \bar Q$ as $k\to\infty$, where $\bar Q$ is defined in Theorem \ref{th:suff}(iii) and is of full rank.
Consequently, there is a constant $\beta_3>0$ such that the inequality
$$
\|P_k d  \|\ge \beta_3 \|d\|,\quad \forall d\in \mathbb R^n
$$
holds for all $k$ sufficiently large. By the mean-value theorem, we can obtain
\begin{eqnarray}\label{est:nabla-Ly}
\|\nabla L(y_k) \| &=& \left \|\; \left( \begin{array}{c} \nabla L(y_k)-\nabla L(\bar y)  \\ \frac{1}{m}\displaystyle \sum _{i=1}^n \Big(e^{m(y_k)_i}-e^{m\bar y_i}\Big)\end{array}\right)\; \right \| \nonumber \\
 &=&  \Big \| P_k(y_k-\bar y)\Big \|\ge \beta_3 \|y_k-\bar y\|,\nonumber
\end{eqnarray}
which together with (\ref{linear-1}) and (\ref{est:grad-L}) implies
\begin{equation}\label{linear-2}
f(y_{k+1}) -f(\bar y) \le f(y_k)-f(\bar y) - \beta _2\beta_3 ^2\gamma_1^{-2} \|y_k-\bar y\|^2.
\end{equation}

On the other hand, again, by Taylor's expansion, there is a $\bar y_k=\bar y+ \bar{\theta}_k(y_k-\bar y)$ with
$\bar{\theta}_k\in(0,1)$ such that
\begin{eqnarray}\label{linear-3}
f(y_k)=L(y_k) &=& L(\bar y)+\nabla L(\bar y)^T(y_k-\bar y) +\frac 12(y_k-\bar y)^T\nabla ^2L(\bar y_k)(y_k-\bar y)\nonumber \\
    &=& f(\bar y) +\frac 12(y_k-\bar y)^T\nabla ^2L(\bar y_k)(y_k-\bar y) .
\end{eqnarray}
Since $\{y_k\}$ is bounded, there must be a constant $\gamma_2>0$ such that
\begin{equation}\label{fk-f-1}
f(y_k)-f(\bar y)\le \gamma _2 \|y_k-\bar y\|^2.
\end{equation}

Now, we estimate a lower bound of $f(y_k)-f(\bar y)$. Denote $d_k=(y_k-\bar y)/\|y_k-\bar y\|$. Then, we have
$$
0= \sum _{i=1}^n \Big(e^{m(y_k)_i} - e ^{m(\bar y)_i}\Big) =  m\Big(e^{[m\check{ y}_k]} \Big) ^T(y_k-\bar y),
$$
where $\check{y}_k:=\bar{y}+\check{\theta}_k(y_k-\bar{y})$ with some $\check{\theta}_k\in(0,1)$, which yields
$$
 \Big(e^{[m\check y_k]} \Big) ^T d_k=0,\quad\forall k\geq0.
$$
Therefore, every accumulation point $\bar d$ of $\{d_k\}$ satisfies $ \Big(e^{[m\bar y]} \Big) ^T \bar d=0$. According to the second order sufficient condition(see Theorem \ref{th:suff}),
there exists a constant $\beta _4>0$  such that
$$
d_k^T\nabla ^2L(\bar y) d_k\ge \beta _4\|d_k\|^2=\beta _4
$$
holds for all $k$ sufficiently large. So there exists a $k_1$ such that
$$
d_k^T\nabla ^2L(\bar y_k) d_k\ge \frac  12 \beta _4,\quad\forall k\ge k_1.
$$
We claim from the last inequality and (\ref{linear-3}) that
$$
f(y_k)- f(\bar y)\ge \beta _4\|y_k-\bar y\|^2,\quad\forall k\ge k_1.
$$
The last inequality together with (\ref{fk-f-1}) yields
\begin{equation}\label{fk-f}
\beta _4\|y_k-\bar y\|^2\le f(y_k)-f(\bar y)\le \gamma _2 \|y_k-\bar y\|^2,\quad\forall k\ge k_1.
\end{equation}
The right-hand side inequality of (\ref{fk-f}) together with (\ref{linear-2}) implies that there exists a $k_0$ such that
$$
0< f(y_{k+1}) -f(\bar y) \le(1- \beta _2\beta_3 ^2\gamma_1^{-2}\gamma_2^{-1})\Big(f(y_k)-f(\bar y) \Big)\stackrel\triangle {=}\rho \Big(f(y_k)-f(\bar y) \Big)
$$
holds for all $k\ge k_0$, where $\rho = 1- \beta _2\beta_3 ^2\gamma_1^{-2}\gamma_2^{-1}\in(0,1)$. Denote $C =f(y_0)-f(\bar y)$.
Then, the last inequality implies
$$
0< f(y_{k+1}) -f(\bar y) \le C \rho ^{k+1},\quad\forall k\ge k_0.
$$
Let $\bar k:=\max\{k_0,k_1\}$. It follows from the last inequality and the left-hand side inequality of (\ref{fk-f}) that, there are two constants $\bar C>0$ and $\bar \rho \in(0,1)$ such that
$$
\|y_k-\bar y\|\le \bar C\bar\rho ^k,\quad \forall k\ge\bar k.
$$
This demonstrates the $R$-linear convergence of $\{y_k\}$.

The remainder of this section is devoted to the $Q$-linear convergence of the sequence $\{x_k\}$.
Denote
$$
\psi(x)= \frac {{\cal A}x^{m-1}\circ x}{{\cal A}x^m}= {\cal A}x^{m-1}\circ \frac {x}{\phi(x)}.
$$
By direct calculation, we have
\begin{eqnarray*}
\psi'(x) &=&  \mbox{diag}\Big({\cal A}x^{m-1} \Big) \Big(\frac {x}{\phi(x)}\Big)'
    +(m-1) \mbox{diag}\Big(\frac {x}{\phi(x)} \Big)  {\cal A}x^{m-2}\\
    &=&  \mbox{diag}\Big({\cal A}x^{m-1} \Big) \Big(\phi(x) ^{-1} I - \frac { x \nabla \phi(x)^T} {\phi(x)^2} \Big)
    +(m-1) \mbox{diag}\Big(\frac {x}{\phi(x)} \Big)  {\cal A}x^{m-2}.
\end{eqnarray*}
At the point of the eigenpair $(\bar x,\bar\lambda)$, it holds that
\begin{eqnarray*}
\psi'(\bar x) &=& \bar\lambda \mbox{diag}\Big(\bar x^{[m-1]} \Big) \Big(\bar \lambda ^{-1} I
    - m \bar\lambda ^{-1} \bar x \Big(\bar x^{[m-1]} \Big)^T \Big)
    +(m-1) \bar\lambda ^{-1} \mbox{diag}(\bar x)  {\cal A}\bar x^{m-2}\\
    &=&  \mbox{diag}\Big(\bar x^{[m-1]} \Big) - m  \bar x ^{[m]} \Big(\bar x^{[m-1]} \Big)^T
        +(m-1) \bar\lambda ^{-1} \mbox{diag}(\bar x)  {\cal A}\bar x^{m-2}\\
    &=& \bar\lambda ^{-1} \mbox{diag}(\bar x) \left [(m-1)  {\cal A}\bar x^{m-2} +\bar\lambda\mbox{diag}\left(\bar x^{[m-2]} \right)
        - m \bar\lambda \bar x ^{[m-1]} \Big(\bar x^{[m-1]} \Big)^T  \right ]\\
    &=& - \bar\lambda ^{-1} \mbox{diag}(\bar x) B(\bar x),
\end{eqnarray*}
where $B(x)$ is defined by (\ref{def:Bx}). Let $g(x):= \psi(x)^{[\frac 1m]}$ and
$$
E_k=\int _0^1 \Big [ g'(\bar x+\tau\,(x_k-\bar x))-g'(\bar x)\Big ] d\,\tau.
$$
Observing that $\psi(\bar x)=\bar x^{[m]}$, we derive from the mean-value theorem that
\begin{eqnarray}\label{error}
 x_{k+1}-\bar x &=& g(x_k)-g(\bar x)= g'(\bar x)(x_k-\bar x) + E_k(x_k-\bar x)\nonumber \\
    &=& \Big(\psi(\bar x)^{[\frac 1m]} \Big)' (x_k-\bar x) +  E_k(x_k-\bar x)\nonumber \\
    &=& \frac 1m \mbox{diag}\Big(\psi(\bar x)^{[\frac 1m-1]}\Big) \psi'(\bar x)(x_k-\bar x) +  E_k(x_k-\bar x)\nonumber \\ 
    &=&  \frac 1m \mbox{diag}\Big(\bar x^{[1-m]}\Big)   \psi'(\bar x)(x_k-\bar x) +  E_k(x_k-\bar x)\nonumber \\
    &=&  -\frac 1m   \bar\lambda ^{-1} \mbox{diag}\Big(\bar x^{[2-m]} \Big) B(\bar x)(x_k-\bar x) +  E_k(x_k-\bar x).
\end{eqnarray}
Multiplying both sides of the last equality by $\mbox{diag}\Big(\bar x^{[\frac{m-2}{2}]} \Big)$, we obtain
\begin{eqnarray}\label{error1}
	&&\mbox{diag}\Big(\bar x^{[\frac{m-2}{2}]} \Big)
\left(	x_{k+1}-\bar x \right) \nonumber \\
&=& \overline{T} \mbox{diag}\Big(\bar x^{[\frac{m-2}{2}]} \Big)(x_k-\bar x) +  {\overline E}_{k}\mbox{diag}\Big(\bar x^{[\frac{m-2}{2}]} \Big)(x_k-\bar x),
\end{eqnarray}
where
\begin{equation}\label{def:T}
	\overline T =   -\frac 1m   \bar\lambda ^{-1} \mbox{diag}\Big(\bar x^{[-\frac{m-2}{2}]} \Big) B(\bar x)  \mbox{diag}\Big(\bar x^{[-\frac{m-2}{2}]} \Big)
\end{equation}
and
$$
{\overline E}_{k} =\mbox{diag}\Big(\bar x^{[\frac{m-2}{2}]} \Big)E_k \mbox{diag}\Big(\bar x^{[-\frac{m-2}{2}]} \Big).
$$

The following theorem establishes the $Q$-linear convergence of the sequences $\{x_k\}$ and $\{\lambda_k\}$.
\begin{theorem}\label{th:Q-linear}
Let $\overline T$ be defined by (\ref{def:T}). The following statements are true.
\begin{description}
  \item [(i)] The spectral radius $\sigma(\overline T)$ of $\overline T$ satisfies $\sigma(\overline T)<1$.
  \item [(ii)] The sequence $\{x_k\}$ $Q$-linearly converges to $\bar x$. That is, there is a constant $\delta \in(0,1)$ such that
     $$
      \|x_{k+1}-\bar x\|_{\bar x} \le \delta \|x_k-\bar x\|_{\bar x}
     $$
     holds for all $k$ sufficiently large, where the vector norm $\|\cdot \|_{\bar x}$ is defined by
     $$
     \|p\|_{\bar x}^2= p^T\mbox{diag}\left(\bar x^{[m-2]} \right) p,\quad\forall p\in \mathbb R^n.
     $$
  \item [(iii)] The sequence of eigenvalue estimations $\{\lambda_k\}$ $Q$-linearly converges to $\bar \lambda$.
\end{description}
\end{theorem}
\begin{proof}
(i) By Proposition \ref{prop:psd} and (\ref{Hession}), it is clear that the matrix $-B(\bar x)$ is positive semidefinite.
Since $\bar x>0$, by the definition of $\overline T$, it is easy to show that the eigenvalues of $\overline T$ are
nonnegative. It suffices to show that every eigenvalue is strictly less than one.

Let   $L(x,\lambda)$ be the Lagrangian function defined by (\ref{def:Lx}). Its Hessian is given by
$$
\nabla _x^2 L(x,\lambda)= -m(m-1){\cal A}x^{m-2} + m(m-1) \lambda  \mbox{diag}\Big(x^{[m-2]} \Big).
$$

By the definition of $B(x)$, we have
\begin{eqnarray*}\label{barT}
\overline T &=&   \frac 1m \bar\lambda^{-1}\mbox{diag}\Big(\bar x^{[\frac{2-m}{2}]} \Big)  \Big [(m-1)  {\cal A}\bar x^{m-2}
    +\bar\lambda\mbox{diag}\left(\bar x^{[m-2]} \right) \nonumber \\
    &&- m \bar\lambda \bar x ^{[m-1]} \Big(\bar x^{[m-1]} \Big)^T  \Big ] \mbox{diag}\Big(\bar x^{[\frac{2-m}{2}]} \Big)\nonumber \\
    &=& I - \frac 1m \bar\lambda^{-1}\mbox{diag}\Big(\bar x^{[\frac{2-m}{2}]} \Big)  \Big [
    	\frac 1m \nabla _x^2 L(\bar x, \bar \lambda) \nonumber \\
        &&+m \bar\lambda \bar x ^{[m-1]} \Big(\bar x^{[m-1]} \Big)^T  \Big ] \mbox{diag}\Big(\bar x^{[\frac{2-m}{2}]} \Big) \nonumber \\
    &\stackrel\triangle {=} & I - \frac 1m \bar\lambda^{-1}\mbox{diag}\Big(\bar x^{[\frac{2-m}{2}]} \Big) \overline W \mbox{diag}\Big(\bar x^{[\frac{2-m}{2}]} \Big) ,
\end{eqnarray*}
where
$$
\overline W=\frac 1m \nabla _x^2L(\bar x, \bar \lambda)  + m \bar \lambda \bar  x^{[m-1]} \Big(\bar x^{[m-1]} \Big)^T.
$$
To verify $\sigma(\overline T)<1$, it suffices to prove that every eigenvalue of $\overline W$ is positive, i.e.,  $\overline W$ is positive definite.

From item (ii) of \cref{th:suff}, we known that $\bar\lambda$ is the largest eigenvalue of the nonnegative irreducible matrix $\bar A$, where
$$
\bar A = \mbox{diag}\Big(\bar x^{[-\frac{m-2}{2}]} \Big) {\cal A}\bar x^{m-2} \mbox{diag}\Big(\bar x^{[-\frac{m-2}{2}]} \Big).
$$
Hence, $-\bar A + \bar \lambda I$ is positive semidefinite and
$$\nabla _x^2 L(\bar x, \bar \lambda) = m(m-1)\mbox{diag}\Big(\bar x^{\frac{[m - 2]}{2}} \Big) \left(-\bar A + \bar \lambda I\right) \mbox{diag}\Big(\bar x^{\frac{[m - 2]}{2}} \Big) $$
is also positive semidefinite.
The second order sufficient condition says that
$$
d^T\nabla _x^2L(\bar x, \bar \lambda) d>0,\quad\forall  d\neq 0:\; \Big(\bar x^{[m-1]} \Big)^Td=0.
$$
For any $d\in \mathbb{R}^{n}\backslash\{0\}$, if $d^T {\bar x}^{[m-1]} = 0$, then,
\begin{eqnarray*}
d^\top \overline  W d &=&  \frac 1m d^T \nabla _x^2 L(\bar x, \bar \lambda) d + m \bar \lambda \left(d^T {\bar x}^{[m-1]} \right)^2\\
			&=&\frac 1m d^T \nabla _x^2 L(\bar x, \bar \lambda) d > 0,
\end{eqnarray*}
and if $d^T {\bar x}^{[m-1]} \neq 0$, then,
\begin{eqnarray*}
	d^\top \overline  W d &=&  \frac 1m d^T \nabla _x^2 L(\bar x, \bar \lambda) d + m \bar \lambda \left(d^T {\bar x}^{[m-1]} \right)^2\\
	&> & m \bar \lambda \left(d^T \bar x^{[m-1]} \right)^2 > 0.
\end{eqnarray*}
Therefore, the matrix $\overline W$ is positive definite.

(ii) Since $\{\overline E_k\}\to 0$ as $k\rightarrow\infty$, there is a constant $\delta \in(0,1)$ such that
inequality $\sigma(\overline T)+\|\overline E_k\| \le \delta $. It follows from (\ref{error1})  that
\begin{eqnarray*}
\|x_{k+1}-\bar x\|_{\bar x} &=& \left \| \mbox{diag}\Big(\bar x^{[\frac {m-2}{2}]} \Big)(x_{k+1}-\bar x) \right \| \nonumber \\
    &\le&   \left \| \overline T \cdot  \mbox{diag}\Big(\bar x^{[\frac {m-2}{2}]}\Big)(x_k-\bar x) \right \|
    +\|\overline E_k\| \cdot\|x_k-\bar x\|_{\bar x}\\
    &\le & \left [ \sigma \Big(\overline T \Big) + \| \overline E_k\| \right ]\|x_k-\bar x\|_{\bar x}
     \le  \delta \|x_k-\bar x\|_{\bar x}.
\end{eqnarray*}
This indicates the validity of conclusion in item (ii).

(iii) Observe that for any $k$, $x_k\in \mathbb B_+$ is feasible. By Taylor's expansion, we get
\begin{eqnarray}\label{temp:th4.1}
-\lambda _{k+1} &=&  L(x_{k+1},\lambda _k)\nonumber \\
    &=& L(x_k,\lambda_k) +\nabla _xL(x_k,\lambda_k)^T s_k
    +\frac 12(x_{k+1}-x_k)^T\nabla _xL(\tilde x_k,\lambda_k)s_k \nonumber \\
    &=& -\lambda_k  +\nabla _xL(x_k,\lambda_k)^T s_k
    +\frac 12 s_k^T\nabla _x^2L(\tilde x_k,\lambda_k)s_k\nonumber \\
    &=&  -\lambda_k  +\nabla _xL(x_k,\lambda_k)^T s_k
    +\frac 12 s_k^T\nabla _x^2L(\bar x,\bar \lambda) s_k +o(\|s_k\|^2) ,
\end{eqnarray}
where $s_k= x_{k+1}-x_k$ and $\tilde x_k=x_k+\theta _k(x_{k+1}-x_k)$ with $\theta_k\in(0,1)$.

By the use of (\ref{gradient}), we get from the first equality in (\ref{KKT-H}),
\begin{eqnarray*}
0 &=& -m\frac {{\cal A}x_k^{m-1}\circ x_k}{{\cal A}x_k^m} + m x_{k+1}^{[m]}\\
    &=&  m\lambda_k ^{-1}\Big(-{\cal A}x_k ^{m-1}\circ x_k +\lambda_k \cdot x_k^{[m]} \Big) + m \Big(x_{k+1}^{[m]}-x_k^{[m]}\Big)\\
    &=& \lambda_k^{-1} \nabla _xL(x_k,\lambda_k)\circ x_k + m \Big(x_{k+1}^{[m]}-x_k^{[m]}\Big),
\end{eqnarray*}
which further indicates
$$
\nabla _xL(x_k,\lambda_k)=-m\lambda_k\mbox{diag}(x_k^{[-1]}) \Big(x_{k+1}^{[m]}-x_k^{[m]}\Big).
$$
Define $\frac {y}{x}\in \mathbb R^n_{++}$ for any $x,y\in \mathbb R^n_{++}$ with elements $(\frac {y}{x})_i=\frac {y_i}{x_i}$ for all $i\in[n]$. It is easy to see that
$$
x_{k+1}^{[m]}-x_k^{[m]}=x_k^{[m]} \circ \left( \frac {x_{k+1}^{[m]}}{x_k^{[m]}} - {\bf e} \right)
=\mbox{diag}\left(x_k^{[m-1]} \circ \sum_{j=0}^{m-1} \frac {x_{k+1}^{[j]}}{x_k^{[j]}} \right)\cdot s_k.
$$
Hence, we obtain
$$
\nabla _xL(x_k,\lambda_k)= -m\lambda_k \mbox{diag}\left(x_k^{[m-2]} \circ \sum_{j=0}^{m-1} \frac {x_{k+1}^{[j]}}{x_k^{[j]}} \right)\cdot s_k
\stackrel\triangle {=}-m\lambda_k M_k s_k,
$$
where
$$
M_k=\mbox{diag}\left(x_k^{[m-2]} \circ \sum_{j=0}^{m-1} \frac {x_{k+1}^{[j]}}{x_k^{[j]}} \right).
$$
Taking limits in both sides of the above equality, we obtain that
$$
\lim _{k\to \infty} M_k = m\cdot \mbox{diag}\Big(\bar x^{[m-2]}\Big).
$$
Consequently, there are two positive constants $C\ge c>0$ such that
\begin{equation}\label{est:nablaL}
c\|s_k\|\le \|\nabla _xL(x_k,\lambda_k)\|\le C\|s_k\|.
\end{equation}
Besides, it is also true that
\begin{eqnarray}\label{temp:th4.1a}
\nabla _xL(x_k,\lambda_k)^Ts_k &=& -m\lambda_k s_k^TM_ks_k \nonumber \\
    &=& -m^2\bar{\lambda} s_k^T \mbox{diag}\Big(\bar x^{[m-2]}\Big)s_k +o(\|s_k\|^2).
\end{eqnarray}

On the other hand, from (\ref{bar:B}) we can derive
$$
\nabla _x^2L(\bar x,\bar\lambda)= mB(\bar x) +m^2\bar\lambda\cdot \mbox{diag}\Big(\bar x^{[m-2]}\Big)
    -m^2\bar{\lambda}\bar x ^{[m-1]} \Big(\bar x ^{[m-1]} \Big)^T.
$$
It then follows from (\ref{temp:th4.1}) and (\ref{temp:th4.1a}) that
\begin{eqnarray*}
 -\lambda _{k+1} &=&  -\lambda_k  -m^2\bar{\lambda} s_k^T \mbox{diag}\Big(\bar x^{[m-2]}\Big)s_k +\frac 12m s_k^T \Big [ B(\bar x) \\
    && +m\bar\lambda\cdot \mbox{diag}\Big(\bar x^{[m-2]}\Big)  -m\bar\lambda \bar x ^{[m-1]} \Big(\bar x ^{[m-1]} \Big)^T\Big ]  s_k +o(\|s_k\|^2) \\
    &\le & -\lambda_k  -\frac 12 m^2\bar\lambda s_k^T \left [ \mbox{diag}\Big(\bar x^{[m-2]}\Big)
    + \bar x ^{[m-1]} \Big(\bar x ^{[m-1]} \Big)^T\right ]  s_k +o(\|s_k\|^2) ,
\end{eqnarray*}
where the last inequality holds due to the fact that $B(\bar x)$ is negative semi-definite. Since $\bar x>0$ and $\{s_k\}\to 0$ as $k\to \infty$, the last inequality shows  that there is a constant $\eta>0$ such that the inequality
$$
\bar\lambda - \lambda_{k+1}\le \bar\lambda -\lambda_k -\eta \|s_k\|^2\le  \bar\lambda -\lambda_k -\eta c^{-1} \|\nabla _xL(x_k,\lambda_k)\|^2
$$
holds for all $k$ sufficiently large, where the last inequality follows from (\ref{est:nablaL}).

Similar to the proof of (\ref{est:nabla-Ly}), it is not difficult to show that there is a constant $c_1>0$ such that
$$
\|\nabla _xL(x_k,\lambda_k)\| \ge c_1\|x_k-\bar x\|
$$
holds for all $k$ sufficiently large. Taking into account that
$x_k$ and $\bar x$ are feasible and $\nabla _xL(\bar x,\bar\lambda)=0$, by Taylor's expansion,
there is a $\bar x_k=\bar x+\bar\theta_k(x_k-\bar x)$ with $\bar\theta_k\in(0,1)$ such that
$$
0<\bar\lambda -\lambda_k= L(x_k,\bar\lambda)- L(\bar x,\bar\lambda)
=\frac 12(x_k-\bar x)^T\nabla _x^2L(\bar x_k,\bar\lambda)(x_k-\bar x) \le C_1\|x_k-\bar x\|^2
$$
holds for all $k$ with some constant $C_1>0$. As a result, we get
\begin{eqnarray*}
0 &< & \bar\lambda - \lambda_{k+1} \le  \bar\lambda -\lambda_k -\eta c^{-1} \|\nabla _xL(x_k,\lambda_k)\|^2 \\
    &\le &  \bar\lambda -\lambda_k -\eta c^{-1}c_1^2 \|x_k-\bar x\|^2\\
    &\le & \bar\lambda -\lambda_k -\eta c^{-1}c_1^2 C_1^{-1}(\bar\lambda-\lambda_k)
    = \rho(\bar\lambda -\lambda_k),
\end{eqnarray*}
where $\rho =1-\eta c^{-1}c_1^2 C_1^{-1} \in(0,1)$. The last inequality yields the $Q$-linear convergent of the sequence $\{\lambda_k\}$.

The proof is complete.
\end{proof}

\section{An improvement of the power-like method}\label{improv}
\setcounter{equation}{0}

In this section, we make an improvement to the power-like method. We first treat the issue of iteration for the method from a new view point. Denote
$$
\bar x^{[m]}= \frac {{\cal A}x_k^{m-1}\circ x_k}{{\cal A}x_k^m},\quad d_k= \bar x_k^{[m]}-x_k^{[m]}
$$
and
$$
x_k^{[m]}(\alpha)= x_k^{[m]} +\alpha d_k =\bar x_k^{[m]} +(\alpha -1) d_k.
$$
Obviously, $x_k(1)=\bar x_k$ is the next iterate determined by the power-like method. In other words, the point $x_{k+1}^{[m]}$ generated by the power-like method can be regarded as the point obtained by starting from $x_k^{[m]}$ along the direction $d_k$ with a unit steplength. Generally speaking, for an optimization method, a larger steplength would improve the efficiency of the method. In what follows, we introduce a line search technique in Algorithm \ref{alg1} to enlarge the steplength.

Rewritten $x_k^{[m]}(\alpha) = \alpha \bar{x}_k^{[m]} +(1-\alpha) x_k^{[m]}.$
Since both $x_k^{[m]}$ and $\bar x_k^{[m]}$ are positive, it is obvious that $x_k^{[m]}(\alpha)>0$ for all  $\alpha>1$ and sufficiently close to $1$. 
In this case, we let  $y_k(\alpha)=\log x_k(\alpha)$. 
It is not difficult to see that $(y_k(\alpha), \mu_{k}(\alpha))$ with $\mu_{k}(\alpha) = m$ is the unique solution of the following system
$$
\alpha \nabla f(y_k) +(\alpha -1) m e^{[my_k]} + \mu e^{[my]} =0\quad \mbox{and} \quad \sum_{i=1}^{m} e^{my_{i}} = 1,
$$
which is the KKT system of the problem
\begin{equation}\label{opt-yy}
\min _{y\in \tilde{\mathbb B}} \Big(\alpha \nabla f(y_k) +(\alpha -1) m e^{[my_k]} \Big)^T(y-y_k),
\end{equation}
where $
\tilde{\mathbb B}=
\{ y\in R^n\; |\; x=e^{[y]} \in \mathbb{B}_{+}\}.$
Hence, $y_k(\alpha)$ is a solution of the problem (\ref{opt-yy}).
In addition, $\bar y_k:=y_k(1)= \log x_k(1) =\log \bar x_k$ is the solution of the problem
$$
\min _{y\in \tilde{\mathbb B}} \nabla f(y_k)^T(y-y_k),
$$
and satisfies
$$
f(\bar y_k) \le f(y_k) + \nabla f(y_k)^T(\bar y_k-y_k)<f(y_k).
$$
As a result, with a given constant $\sigma \in(0,1)$, the following inequality
\begin{equation}\label {search-2}
f(y_k(\alpha_k)) \le f(y_k) +\sigma [\alpha_k\nabla f(y_k)+(\alpha_k-1)me^{[my_k]}]^T(y_k(\alpha_k)-y_k)
\end{equation}
holds for all $\alpha>1$ and sufficiently close to $1$.
Thus, a steplength $\alpha_k>1$ can be obtained by 
a back tracking process.
For some $\beta >0$ and $\rho \in(0,1)$, let $\alpha_k=1+\beta \rho ^i$, where $i=0,1,\ldots$, be the largest scalar such that $x_k^{[m]}(\alpha_k)>0$ and the condition (\ref{search-2}) is satisfied. 

With the above preparation,  we proposed an ``improved power-like method" in \cref{alg2}.
\begin{algorithm}[!tbph]
\caption{(Improved power-like method with a line search)}\label{alg2}
\begin{algorithmic}[1]
\STATE Given  parameters $\delta, \rho, \sigma \in(0,1)$ and an initial point $x_0\in\mathbb{B}_{+}$. Let $k:=0$.
\WHILE{$F(x_k)\neq 0$}

\STATE Compute $\bar x^{[m]}= \frac {{\cal A}x_k^{m-1}\circ x_k}{{\cal A}x_k^m}$ and $d_k= \bar x_k^{[m]}-x_k^{[m]}$. Let
\begin{equation}\label{iter-search}
x_k^{[m]}(\alpha)= \bar x_k^{[m]} +(\alpha -1) d_k,\quad
y_k(\alpha) =\log x_k(\alpha).
\end{equation}
\STATE Determine $\alpha_k$ be the largest element of 
$\{1+\beta \rho^{i}\mid i=0,1,\ldots\}$ satisfying the inequality 
\begin{equation}\label {search-1}
x_k^{[m]}(\alpha_k)\geq\delta \bar{x}_k^{[m]}.
\end{equation}
and the line search condition (\ref{search-2}).
\STATE Let $x_{k+1}:=x_k(\alpha_k)$ and $k:=k+1$.
\ENDWHILE
\end{algorithmic}
\end{algorithm}

\begin{remark}\label{imp mth}
{\rm(i)} It is easy to see that both  conditions (\ref{search-2}) and (\ref{search-1}) hold for all $\alpha_k$ closing to $1$. Consequently, Algorithm \ref {alg2} is well-defined. {\rm(ii)} The condition (\ref{search-2}) leads to the fact that the sequence $\{f(y_k)\}$ is descending. {\rm(iii)} The condition (\ref{search-1}) ensures that the sequence $\{x_k\}$ has a positive lower bound.
\end{remark}
The following theorem establishes the convergence of the improved power-like method. 
\begin{theorem}
The sequences $\{\lambda_k\}$ and $\{x_k\}$ generated by Algorithm \ref{alg2} converge to the spectral radius of $\cal A$ and its corresponding positive eigenvector, respectively.
\end{theorem}

\begin{proof}
By using (\ref{search-1}), we can prove in a way similar to \cref{th:positive-H} that the sequence $\{x_k\}$ has a positive lower bound.

Suppose that for an infinite set $K$, the subsequence $\{x_k\}_K$ converges to some $\bar x>0$ and the subsequence $\{x_{k+1}\}_K$ converges to some $\tilde x>0$. Accordingly, the subsequences $\{y_k\}_K$ and $\{y_{k+1}\}_K$ converge to some $\bar y=\log \bar x$ and $\tilde y=\log \tilde x$, respectively.
The line search condition (\ref{search-2}) means
$$
f(y_{k+1})-f(y_k)  \le \sigma[\alpha_k\nabla f(y_k)-(\alpha_k-1)me^{[my_k]}]^T(y_{k+1}-y_k)\le 0 .
$$
By taking limits in both sides of the last inequalities as $k\to\infty$ with $k\in K$, we obtain
$$
[\alpha\nabla f(\bar{y})-(\alpha-1)me^{[m\bar{y}]}]^T(\tilde{y}-\bar{y})=0,
$$
where $\lim_{k\in K,~k\rightarrow\infty}\alpha_k=\alpha$. Besides, taking limits in the KKT condition of system (\ref{opt-yy}), we can get
$$
\alpha\nabla f(\bar y) +(\alpha-1)m e^{[m\bar y]}+m e^{[m\tilde y]}=0\quad\mbox{and}\quad \sum _{i=1}^n e^{m\tilde y_i}=1.
$$
The last two equalities show that the point $\tilde y$ together with the  multiplier $\mu =m$ is the unique KKT point of the problem
$$\left\{\begin{array}{lll}
\min & [\alpha\nabla f(\bar y)+(\alpha-1)m e^{[m\bar y]}]^T(y-\bar y),  \\
\mbox{s.t.} & \displaystyle \sum _{i=1}^n e^{m y_i}=1.
\end{array}\right.
$$
Consequently, $\tilde y$ is a global solution of the above problem. However,  the fact
$$
[\alpha\nabla f(\bar{y})-(\alpha-1)me^{[m\bar{y}]}]^T(\tilde{y}-\bar{y})=0\quad\mbox{and}\quad \sum _{i=1}^n e^{m\bar y_i}=1
$$
shows that $\bar y$ is also a solution of the last problem, and hence, it together with $\mu=m$ is a KKT point of the last problem. In other words, $\bar y$ satisfies
$$
\nabla f(\bar y) + m e^{[m\bar y]}=0\quad\mbox{and}\quad \sum _{i=1}^n e^{m\bar y_i}=1,
$$
which means that $\bar y$ is a KKT point of the problem (\ref{H-opt}). By Theorem \ref{th:kkt-H}, $\bar x=e^{[\bar y]}$ is a positive eigenvector of $\cal A$ corresponding to the positive eigenvalue $\bar\lambda ={\cal A}\bar x^m$.
\end{proof}

\section{Numerical results}\label{numer}
\setcounter{equation}{0}

In this section, we  do some numerical experiments to test the efficiency of the proposed power-like method and
its improvement version. In the course of specific experiments, we add the idea of the well-known BB method for solving unconstrained optimization problems to the progress of line search. In detail, we choose the constant $\beta$ similar to the so-called BB step in the gradient method for unconstrained optimization.

Clearly, any eigenvector of $\cal A$ is a solution of the nonlinear equation
$$
F(x)=-{\cal A}x^{m-1}+{\cal A}x^m\cdot x^{[m-1]}=0.
$$
Let $\lambda_k={\cal A}x_k^{m}$. Then, we have
$$
d_k=\bar{x}_k^{[m]}-x_k^{[m]}=\frac{{\cal A}x_k^{m-1}\circ x_k}{\lambda_k}-x_k^{[m]}=-\lambda_k^{-1}D_kF(x_k),
$$
where $D_k=\mbox{diag}(x_k)$. We  rewrite the iteration scheme (\ref{iter-search}) as
\begin{equation}\label{iter-search2}
x_k^{[m]}(\alpha)=x_k^{[m]}-\alpha\lambda_k^{-1}D_kF(x_k).
\end{equation}
Considering our purpose is to find the positive eigenpair of the irreducible nonnegative tensor ${\cal A}$, by taking a variable transform $z=x^{[m]}$, we can get the following equivalent formulation to the last problems
$$
G(z):=F\Big(z^{[\frac 1m]}\Big) =  -{\cal A}\Big(z^{[\frac 1m]}\Big)^{m-1}+{\cal A}\Big(z^{[\frac 1m]}\Big)^m \cdot z^{[\frac {m-1}m]}=0.
$$
Thus, the iterative scheme (\ref{iter-search2}) can be further rewritten as
$$
z_{k+1}(\alpha)=z_k -\alpha\lambda _k ^{-1} D_k G(z_k) \quad \mbox{where} \quad D_k=\mbox{diag}(x_k)=\mbox{diag}\Big(z_k^{[\frac 1m]}\Big).
$$

In other words, the scheme (\ref{iter-search}) can be regarded as an iterative method for solving nonlinear equation $G(z)=0$. Now, we will choose the parameter $\beta$ in the process of line search in a way similar to the BB (or spectral) method. Here, we provide two strategies, which are
$$
\beta _{1k} =\lambda _k \frac {t_k^TD_ks_k}{\|D_kt_k\|^2}-1\quad\mbox{and}\quad
\beta _{2k} =\lambda _k \frac {t_k^Ts_k}{t_k^TD_kt_k}-1,
$$
where $s_k:=z_k-z_{k-1}$ and $t_k:=G(z_k)-G(z_{k-1})$. BB steps $\beta_{1k}$ and $\beta_{2k}$ are both bounded, and the convergence of Algorithm \ref{alg2} with $\beta=\beta_{1k}$ or $\beta_{2k}$ will remain unchanged.
In addition, if $\beta_{1k}$ and $\beta_{2k}$ are both negative, then, we will simply set $\beta=0$ or $\beta=1$.

We tested the performance of the proposed power-like method and its improved versions with $\beta=\beta_{1k}$ and $\beta_{2k}$.
The related methods are denoted by ``Improved method1" and ``Improved method2", respectively.
The numerical experiments were done with MATLAB R2022a on a personal laptop with Intel(R) Core(TM) CPU i7-10510U @1.80GHz and 16 GB memory running Microsoft Windows 11. While doing numerical experiments, the tensor toolbox \cite{TensorT} was employed to proceed tensor computation.

We compared our methods with the well-known higher-order power method in \cite{LP16}. In the process of experiments, instead of solving the original problem, we solve the scaled system $\bar{{\cal A}}x^{m-1}=\lambda x^{[m-1]}$, where $\bar{{\cal A}}=\frac{1}{a}{\cal A}$ and $a$ is the largest entry of ${\cal A}$. After finding the $H$-eigenpair $(\bar{\lambda}^*,x^*)$ of $\bar{{\cal A}}$, $(\lambda^*,x^*)$ with $\lambda^*=a\bar{\lambda}^*$ is the desired eigenpair for ${\cal A}$. For all test problems and any algorithms, the initial point $x_0\subset\mathbb{B}_+$ is randomly generated, and the termination criterion is set to
\begin{equation}\label{stop}
Res:=\|\bar{{\cal A}}x^{m-1}_k-\bar{{\cal A}}x^m_k\cdot x_k^{[m-1]}\|\leq10^{-8}.
\end{equation}
If the above stop condition is not met but the number of iterations has reached to 200, we will also stop the iterative process and regard this situation as failing to address the problem.

To show how algorithms behave, for each problem, we tested all the methods on each  problem with 100 different initial points. The testing indices including the average number of iterations denoted by ``iter", the average computing time in seconds denoted by ``Cpu", the average residual given by (\ref{stop}) denoted by ``Res", and the rate of successful terminations denoted by ``suc\%", will be used to indicate the efficiency of the algorithms.

First, we tested three irreducible nonnegative tensors which may not necessarily be symmetric.

\noindent{\bf Problem 1.}(\cite{CZ13}) Let tensor ${\cal A}\in\mathbb{S}^{[4,2]}$ be an irreducible nonnegative tensor with
$$
a_{1111}=a_{2222}=\frac{4}{\sqrt{3}},\quad a_{1112}=a_{1121}=a_{1211}=a_{2111}=1,
$$
$$
a_{1222}=a_{2122}=a_{2212}=a_{2221}=1
$$
and $a_{i_1i_2i_3i_4}=0$ elsewhere.

\noindent{\bf Problem 2.}(\cite{NQZ09}) Let tensor ${\cal A}=[A(1,:,:),A(2,:,:),A(3,:,:)]\in\mathbb{R}^{[3,3]}$
be irreducible nonnegative with 
$$
A(1,:,:)=\left(
\begin{array}{lll}
6.48 & 8.35 & 1.03\\
4.04 & 3.72 & 1.43\\
6.61 & 6.41 & 1.35
\end{array}
\right),~~
A(2,:,:)=\left(
\begin{array}{lll}
9.02 & 0.78 & 6.90\\
9.70 & 4.79 & 1.85\\
2.09 & 4.17 & 2.98
\end{array}
\right),
$$
and
$$
A(3,:,:)=\left(
\begin{array}{lll}
9.55 & 1.57 & 6.89\\
5.63 & 5.55 & 1.45\\
5.65 & 8.29 & 6.22
\end{array}
\right).
$$

\noindent{\bf Problem 3.}(\cite{CZ13}) Let ${\cal A}\in\mathbb{R}^{[4,2]}$ be an irreducible nonnegative tensor with
$$
a_{1112}=30,~~a_{1212}=1,~~a_{1222}=1,~~a_{2111}=6,~~a_{2112}=13,~~a_{2122}=37,
$$
and $a_{i_1i_2i_3i_4}=0$ elsewhere.

Tables \ref{tab1}-\ref{tab3} show the average performance of the tested methods on the above three problems.
\begin{table}[!htbp]
	\caption{Numerical results for Problem 1}\label{tab1}
	{\scriptsize
		\def\temptablewidth{1\textwidth}
		\begin{tabular*}{\temptablewidth}{@{\extracolsep{\fill}}cccccccc}\toprule
			Alg. & \textsf{$\bar{\lambda^*}$} & \textsf{iter}& \textsf{Cpu} & \textsf{Res}& \textsf{suc\%} \\ \midrule
			Power method & 2.73205 &  20 & 0.00168 & 7.15e-09 & 100 \\
            \hline
            Power-like method & 2.73205 &  30 & 0.00255 & 7.32e-09 & 100 \\
            \hline
            Improved method1 & 2.73205 & 6.2 & 0.00065 & 1.1e-09 & 100 \\
            \hline
            Improved method2 & 2.73205 & 6.3 & 0.00067 & 4.32e-10 & 100 \\
			\bottomrule
	\end{tabular*}}
\end{table}

\begin{table}[!htbp]
	\caption{Numerical results for Problem 2}\label{tab2}
	{\scriptsize
		\def\temptablewidth{1\textwidth}
		\begin{tabular*}{\temptablewidth}{@{\extracolsep{\fill}}cccccccc}\toprule
			Alg. & \textsf{$\bar{\lambda^*}$} & \textsf{iter}& \textsf{Cpu} & \textsf{Res}& \textsf{suc\%} \\ \midrule
			Power method & 4.45951 & 9.8 & 0.00072 & 4.54e-09 & 100 \\
            \hline
            Power-like method & 4.45951 &  23 & 0.00152 & 6.19e-09 & 100 \\
            \hline
            Improved method1 & 4.45951 & 9.1 & 0.00076 & 2.3e-09 & 100 \\
            \hline
            Improved method2 & 4.45951 & 9.2 & 0.00078 & 2.51e-09 & 100 \\
			\bottomrule
	\end{tabular*}}
\end{table}

\begin{table}[!htbp]
	\caption{Numerical results for Problem 3}\label{tab3}
	{\scriptsize
		\def\temptablewidth{1\textwidth}
		\begin{tabular*}{\temptablewidth}{@{\extracolsep{\fill}}cccccccc}\toprule
			Alg. & \textsf{$\bar{\lambda^*}$} & \textsf{iter}& \textsf{Cpu} & \textsf{Res}& \textsf{suc\%} \\ \midrule
			Power method & 1.10824 &  11 & 0.00108 & 4.11e-09 & 100 \\
            \hline
            Power-like method & 1.10824 &  20 & 0.00174 & 6.16e-09 & 100 \\
            \hline
            Improved method1 & 1.10824 & 7.4 & 0.00082 & 4.65e-10 & 100 \\
            \hline
            Improved method2 & 1.10824 & 7.4 & 0.00084 & 3.7e-10 & 100 \\		
			\bottomrule
	\end{tabular*}}
\end{table}

From Tables \ref{tab1}-\ref{tab3} we can see the following features.
\begin{itemize}
  \item For Problems 1-3, starting from any initial point in the feasible set, our proposed power-like method and its improved versions can always find the spectral radius of the involved tensor whatever it is symmetric or asymmetric.
  \item In terms of computation time and iterations, the Improved method1 and Improved method2 generally need less iteration numbers as well as computation time compared with the power method. Especially for Problem 1 where the involved tensor is irreducible nonnegative symmetric, the average number of iterations used by the Improved methods 1 and 2 are almost one-third of that of the power method.
\end{itemize}

\noindent{\bf Problem 4.}(\cite{CPZ11}) Let ${\cal A}\in\mathbb{R}^{[3,3]}$ be an irreducible nonnegative tensor with elements
$$
a_{111}=1,~~a_{133}=1,~~a_{211}=1,~~a_{311}=1
$$
and $a_{i_1i_2i_3}=0$ elsewhere.

The results of the tested methods are shown in Table \ref{tab4}.
\begin{table}[!htbp]
\caption{Numerical results for Problem 4}\label{tab4}
{\scriptsize
\def\temptablewidth{1\textwidth}
\begin{tabular*}{\temptablewidth}{@{\extracolsep{\fill}}cccccccc}\toprule
Alg. & \textsf{$\bar{\lambda^*}$} & \textsf{iter}& \textsf{Cpu} & \textsf{Res}& \textsf{suc\%} \\ \midrule
    Power method & - & - &    - & - &  0 \\
    \hline
    Power-like method & 1.41421 &  19 & 0.00084 & 6.05e-09 & 100 \\
    \hline
    Improved method1 & 1.41421 &  19 & 0.00098 & 5.99e-09 & 100 \\
    \hline
    Improved method2 & 1.41421 &  19 & 0.00103 & 6.08e-09 & 100 \\
 \bottomrule
\end{tabular*}}
\end{table}

We can see from Table \ref{tab4} that, the proposed methods can always find the spectral radius of the involved
tensors while the power method failed to find its $H$-eigenpairs within 200 iterations. It is interesting to note
that for this problem, the BB steps were always unavailable and  $\alpha_k=1$ for all $k$. As a result,
the power-like method and all its improved versions behave the same way.

At last, we tested the methods on some high-dimensional problems.

\noindent{\bf Problem 5.}(\cite{CHHZ19}) Let ${\cal B}\in\mathbb{R}^{[m,n]}$ be a nonnegative
tensor whose entries are random numbers uniformly distributed in $[0,1]$. Let ${\cal A}={\cal B}+\delta{\cal I}$,
where ${\cal I}$ is the identity tensor and $\delta>0$ is a parameter. 

\noindent{\bf Problem 6.}(\cite{ZNG20}) The entries of the irreducible nonnegative tensor ${\cal A}\in\mathbb{S}^{[m,n]}$ are
$$
a_{i_ii_2\cdots i_m}=|\tan(i_1)+\tan(i_2)+\cdots+\tan(i_m)|,\quad \forall i_1,\ldots,i_m\in[n].
$$

We tested these four methods on the above two problems with different order $m$'s, dimension $n$'s and parameter $\delta$'s. Tables \ref{tab5} and \ref{tab6} list the average performance of these methods.
\begin{center}
	\begin{table}[!htbp]
		\caption{Numerical results for Problem 5}\label{tab5}
		\setlength\tabcolsep{2pt}
		{\scriptsize
			\def\temptablewidth{1\textwidth}
			\begin{tabular*}{\temptablewidth}{@{\extracolsep{\fill}}c|cccccccccc}\toprule
				& & Power method && Power-like method && Improved method1 && Improved method2 \\
				$(m,n,\delta)$ && \textsf{iter} / \textsf{Cpu}/ \textsf{suc\%} && \textsf{iter} / \textsf{Cpu} / \textsf{suc\%} && \textsf{iter} / \textsf{Cpu}/ \textsf{suc\%} && \textsf{iter} / \textsf{Cpu}/ \textsf{suc\%} \\\midrule
				(3, 20, $10^{2}$) && 18.8 / 0.00095 / 100 && 33.8 / 0.00166 / 100 && 8.8 / 0.00056 / 100 && 8.9 / 0.00055 / 100\\
				(3, 20, $10^{4}$) && - /    - /  0 && - /    - /  0 && 33.6 / 0.00207 / 100 && 38.9 / 0.00234 / 100\\
				(3, 50, $10^{2}$) && 9.2 / 0.00080 / 100 && 23.0 / 0.00204 / 100 && 8.0 / 0.00088 / 100 && 8.0 / 0.00085 / 100\\
				(3, 50, $10^{4}$) && - /    - /  0 && - /    - /  0 && 12.4 / 0.00165 / 100 && 13.0 / 0.00180 / 100\\
				(4, 20, $10^{3}$) && 13.9 / 0.00168 / 100 && 23.0 / 0.00260 / 100 && 8.2 / 0.00109 / 100 && 8.3 / 0.00112 / 100\\
				(4, 20, $10^{5}$) && - /    - /  0 && - /    - /  0 && 29.1 / 0.00378 / 100 && 29.3 / 0.00378 / 100\\
				(4, 50, $10^{3}$) && 7.0 / 0.02210 / 100 && 18.2 / 0.05687 / 100 && 7.6 / 0.02396 / 100 && 7.6 / 0.02418 / 100\\
				(4, 50, $10^{5}$) && 37.4 / 0.11667 / 100 && 52.9 / 0.16577 / 100 && 10.1 / 0.03209 / 100 && 10.2 / 0.03216 / 100\\
				\bottomrule
		\end{tabular*}}
	\end{table}
\end{center}
\begin{center}
	\begin{table}[!htbp]
		\caption{Numerical results for Problem 6}\label{tab6}
        \setlength\tabcolsep{2pt}
		{\scriptsize
			\def\temptablewidth{1\textwidth}
			\begin{tabular*}{\temptablewidth}{@{\extracolsep{\fill}}c|cccccccccc}\toprule
				& & Power method && Power-like method && Improved method1 && Improved method2 \\
				$(m,n)$ && \textsf{iter} / \textsf{Cpu}/ \textsf{suc\%} && \textsf{iter} / \textsf{Cpu} / \textsf{suc\%} && \textsf{iter} / \textsf{Cpu}/ \textsf{suc\%} && \textsf{iter} / \textsf{Cpu}/ \textsf{suc\%} \\\midrule
				(3, 100) && 23.3 / 0.0061 / 100 && 23.6 / 0.0060 / 100 && 12.7 / 0.0037 / 100 && 12.2 / 0.0036 / 100\\
                (3, 200) && 23.3 / 0.0812 / 100 && 24.5 / 0.0860 / 100 && 12.8 / 0.0458 / 100 && 12.8 / 0.0458 / 100\\
                (3, 300) && 23.0 / 0.2802 / 100 && 25.5 / 0.3113 / 100 && 13.4 / 0.1657 / 100 && 13.1 / 0.1605 / 100\\
                (4, 30) && 18.8 / 0.0064 / 100 && 20.0 / 0.0066 / 100 && 11.4 / 0.0039 / 100 && 11.7 / 0.0041 / 100\\
                (4, 60) && 19.1 / 0.1474 / 100 && 21.2 / 0.1627 / 100 && 12.0 / 0.0922 / 100 && 12.3 / 0.0950 / 100\\
                (5, 20) && 16.4 / 0.0356 / 100 && 19.0 / 0.0408 / 100 && 10.9 / 0.0242 / 100 && 11.1 / 0.0246 / 100\\
                (5, 40) && 17.6 / 1.2636 / 100 && 20.1 / 1.4419 / 100 && 11.7 / 0.8337 / 100 && 11.8 / 0.8424 / 100\\
				\bottomrule
		\end{tabular*}}
	\end{table}
\end{center}

From Tables \ref{tab5} and \ref{tab6} we can see the following features.
\begin{itemize}
  \item For Problem 5, when the parameter $\delta$ is relatively small, all the methods performed well. However, with the increasing of $\delta$, the Improved methods 1 and 2 are getting more efficiency than the power method and the power-like method. Moreover, the latter two methods may fail to find the largest $H$-eigenpair within 200 iterations.
  \item For Problem 6, all the tested methods can always find the largest $H$-eigenpair of the tensors successfully. The improved methods generally needed less iteration numbers as well as computation time than the power method did.
\end{itemize}

For clarity,  we also plotted a figure, namely Figure \ref{fig1}, to show the performance of the tested four methods on two cases of Problems 5 and 6.
\begin{figure}[!htbp]
\center{\includegraphics[width=0.45\textwidth]{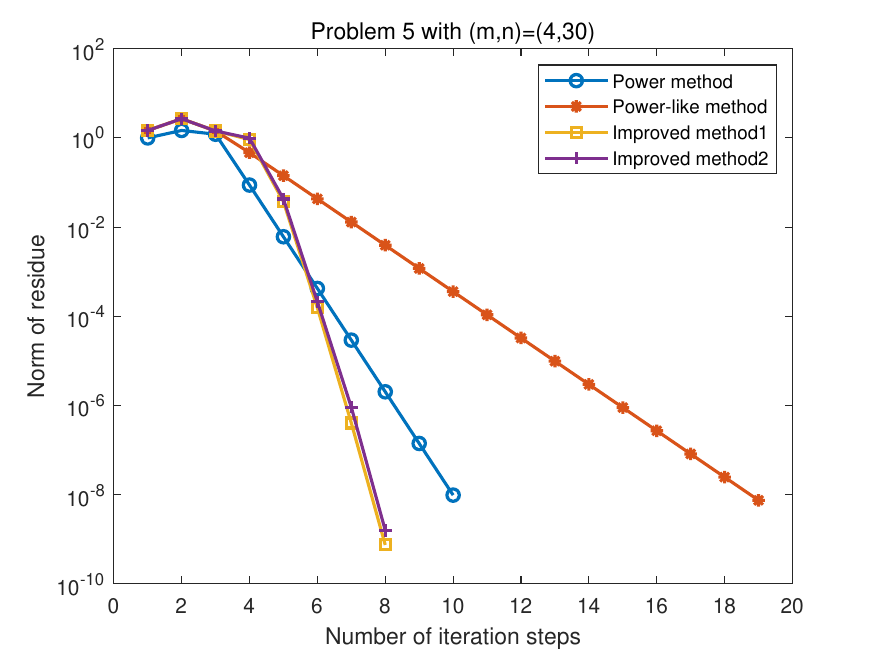} \ \
\includegraphics[width=0.45\textwidth]{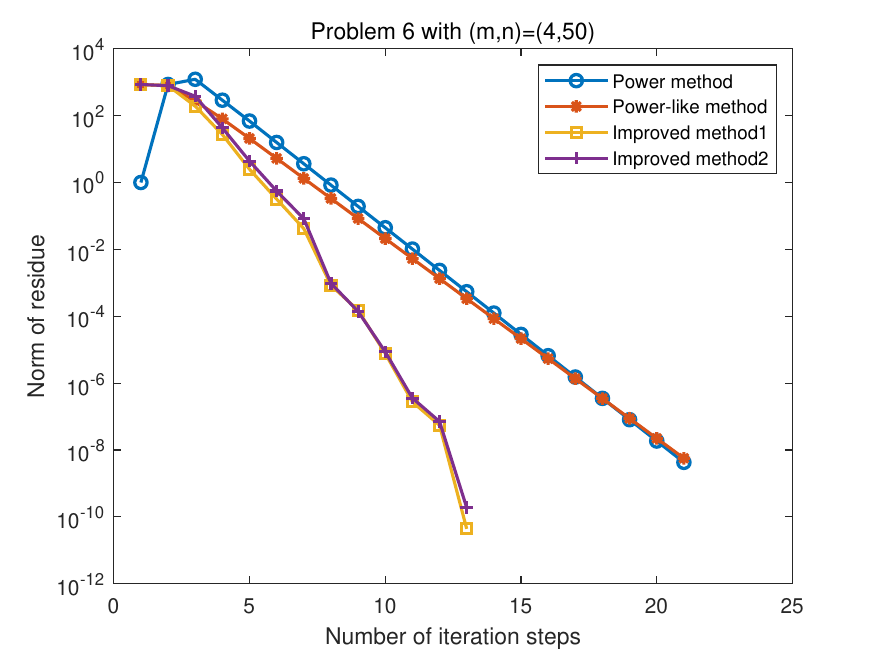}}
\caption{Evolutions of residue with respect to the number of iteration steps}
\label{fig1}
\end{figure}

At the end of this section, we created a table, i.e., Table \ref{tab7}, to show the comparison in the average iteration numbers as well as the CPU time used among our proposed improved methods and the power method. The meaning of each column is given by
$$
\mbox{I1}:=\frac{\mbox{iterations of the IM1}}{\mbox{iterations of the PM}}\times 100\%,\quad
\mbox{I2}:=\frac{\mbox{iterations of the IM2}}{\mbox{iterations of the PM}}\times 100\%,
$$
and
$$
\mbox{T1}:=\frac{\mbox{Cpu time for the IM1}}{\mbox{Cpu time for the PM}}\times 100\%,\quad
\mbox{T2}:=\frac{\mbox{Cpu time for the IM2}}{\mbox{Cpu time for the PM}}\times 100\%,
$$
where ``IM1(2)" represents for improved method1(2) and ``PM" represents for the power method.
\begin{table}[!htbp]
	\caption{Comparison between the improved methods and the power method}\label{tab7}
	{\scriptsize
		\def\temptablewidth{1\textwidth}
		\begin{tabular*}{\temptablewidth}{@{\extracolsep{\fill}}cccccccc}\toprule
			Problem            & I1 & I2 & T1 & T2 \\ \midrule
			Problem 1 & 31.9\%  & 32.1\%  & 38.5\% &  40.2\% \\
			Problem 2 & 92.9\%  &  93\%  & 104\% &  107\% \\
			Problem 3 &  64.6\%  & 64.8\%  & 76.3\% &  78.3\% \\
			Problem 6(3, 100) & 54.4\%  & 52.2\%  & 60.8\% &  59.7\% \\
			Problem 6(3, 200) & 54.9\%  & 54.9\%  & 56.5\% &  56.3\% \\
			Problem 6(3, 300) & 58.4\%  & 56.7\%  & 59.1\% &  57.3\% \\
			Problem 6(4, 30) & 60.4\%  & 62.2\%  & 61.1\% &  64.8\% \\
			Problem 6(4, 60) & 63.0\%  & 64.7\%  & 62.6\% &  64.4\% \\
			Problem 6(5, 20) & 66.7\%  & 67.7\%  & 67.9\% &  69.0\% \\
			Problem 6(5, 40) & 66.1\%  & 66.7\%  & 66.0\% &  66.7\% \\
			\bottomrule
	\end{tabular*}}
\end{table}

\section{Conclusions}\label{conc}

In this paper, we  focused on the numerical algorithm of finding the spectral radius of a nonnegative irreducible symmetric tensor.
By transferring the eigenvalue system into an equivalent DC programming, we designed a cheaper and efficient first-order method called power-like method to find the spectral radius and its corresponding eigenvector of a nonnegative irreducible tensor. The generated sequence of eigenvalue estimates is monotonically increasing and converges to the spectral radius of the tensor. Moreover, both the eigenvalue sequence and the eigenvector sequence are $Q$-linearly convergent. To the best of our knowledge, up to now, there is no other first-order method where the generated sequence of eigenvector estimates can achieve $Q$-linear convergence rate. To improve the efficiency of the power-like method, we introduced a line search technique, and the improved methods possess the same convergence properties as their original version. Furthermore, as observed in numerical experiments, BB step is a pretty good alternative for the initial steplength of the improved methods.

\bibliographystyle{siam}
\bibliography{references}
\end{document}